%% file: main.tex
\documentclass[12pt]{amsart}
\usepackage[margin=1in,lmargin=1in,rmargin=1in]{geometry}
\usepackage{amsmath,latexsym,amssymb,amsthm,graphicx}
\usepackage{mathtools}
\usepackage{caption}
\usepackage{subcaption}
\usepackage{xfrac}
\usepackage{fix-cm}
\usepackage{color}
\usepackage{verbatim}
\usepackage[bookmarks=true,%
    colorlinks=true,%
    linkcolor=blue,%
    citecolor=blue,%
    filecolor=blue,%
    menucolor=blue,%
    urlcolor=blue]{hyperref}
\usepackage{doi}
\usepackage{stmaryrd}
\usepackage{multirow}
\usepackage{float}
\usepackage{tabularx}
\usepackage[shortlabels]{enumitem}
\usepackage{pdfsync}
\usepackage{tikz}
\usepackage{tikz-cd}
\usepackage{float}
\usepackage[utf8]{inputenc}
\usepackage[textsize=small]{todonotes}

\usepackage{hyphenat}
\usepackage{booktabs}

\setenumerate[1]{label=(\arabic*)}
\numberwithin{equation}{section}

\theoremstyle{plain}
\newtheorem{thm}{Theorem}[section]
\newtheorem{prop}[thm]{Proposition}
\newtheorem{lem}[thm]{Lemma}
\newtheorem{cor}[thm]{Corollary}

\theoremstyle{definition}
\newtheorem{defn}[thm]{Definition}
\newtheorem{ex}[thm]{Example}

\theoremstyle{remark}
\newtheorem{remark}[thm]{Remark}

\newtheorem*{acknowledgements}{Acknowledgements}

\newcommand{\thmref}[1]{Theorem~\ref{#1}}
\newcommand{\propref}[1]{Proposition~\ref{#1}}
\newcommand{\secref}[1]{Section~\ref{#1}}

\newcommand{\lemref}[1]{Lemma~\ref{#1}}
\newcommand{\corref}[1]{Corollary~\ref{#1}}
\newcommand{\figref}[1]{Figure~\ref{#1}}
\newcommand{\exref}[1]{Example~\ref{#1}}

\newcommand{\eqnref}[1]{Equation~\eqref{#1}}
\newcommand{\defref}[1]{Definition~\ref{#1}}

\newcommand\Z{\ensuremath{\mathbb{Z}}}

\newcommand\R{\ensuremath{\mathbb{R}}}
\newcommand\C{\ensuremath{\mathbb{C}}}
\def\H{\ensuremath{\mathbb{H}}} 

\newcommand{\calA}{{\mathcal A}}
\newcommand{\calB}{{\mathcal B}}
\newcommand{\calC}{{\mathcal C}}

\newcommand{\calE}{{\mathcal E}}

\newcommand{\calI}{{\mathcal I}}

\newcommand{\calO}{{\mathcal O}}
\newcommand{\calP}{{\mathcal P}}

\newcommand{\calS}{{\mathcal S}}
\newcommand{\calT}{{\mathcal T}}

\DeclareMathOperator{\sysEL}{sysEL}

\DeclareMathOperator{\sys}{sys}

\DeclareMathOperator{\area}{area}
\DeclareMathOperator{\length}{length}

\newcommand{\del}{\partial}

\newcommand{\eps}{\varepsilon}

\DeclareMathOperator{\re}{Re}
\DeclareMathOperator{\im}{Im}
\DeclareMathOperator{\Res}{Res}

\DeclareMathOperator{\PSL}{PSL}

\DeclareMathOperator{\EL}{EL}
\DeclareMathOperator{\ELsys}{sys_{\EL}}
\DeclareMathOperator{\CAT}{CAT}
\DeclareMathOperator{\conf}{Conf}
\DeclareMathOperator{\SR}{SR}
\newcommand{\CHAT}{\widehat{\C}}

\graphicspath{{mpdraws/}{draws/}}

\begin{document}
\title{The extremal length systole of the Bolza surface}

\author[M. Fortier Bourque]{Maxime Fortier Bourque}
\address{
School of Mathematics and Statistics, University of Glasgow, University Place, Glasgow, United Kingdom, G12 8QQ}
\email{maxime.fortier-bourque@glasgow.ac.uk}

\author[D. Martínez-Granado]{Dídac Martínez-Granado}
\address{Department of Mathematics\\
         University of California Davis,
         Davis, California, 95616\\
         USA}
\email{dmartinezgranado@ucdavis.edu}

    \author[F. Vargas Pallete]{Franco Vargas Pallete}
\address{
Department of Mathematics\\
         Yale University\\
New Haven, CT 08540\\
USA}
\email{franco.vargaspallete@yale.edu}

\date{Draft of \today}
\begin{abstract}
We prove that the extremal length systole of genus two surfaces attains a strict local maximum at the Bolza surface, where it takes the value $\sqrt{2}$.
\end{abstract}

\maketitle

\input{intro}
\input{extremal_length}
\input{simple}
\input{hyperelliptic}
\input{explicit}
\input{estimates}
\input{derivatives}

\appendix
\input{landen}

\input{prisms}

\bibliographystyle{hamsalpha}
\bibliography{main}

\end{document}

%% file: intro.tex
\section{Introduction}

Extremal length is a conformal invariant that plays an important role in complex ana\-ly\-sis,  complex dynamics, and Teichm\"uller theory \cite{AhlforsQuasi,AhlforsConf,JenkinsBook}. It can be used to define the notion of quasiconformality, upon which the Teichm\"uller distance between Riemann surfaces is based. In turn, a formula of Kerckhoff \cite[Theorem 4]{Kerckhoff} shows that Teichm\"uller distance is determined by extremal lengths of (homotopy classes of) essential simple closed curves, as opposed to all families of curves.

The \emph{extremal length systole} of a Riemann surface $X$ is defined as the infimum of the extremal lengths of all essential closed curves in $X$. This function fits in the framework of \emph{generalized systoles} (infima of collections of ``length'' functions) developed by Bavard in \cite{Bav:97} and \cite{Bav:05}. In contrast with the hyperbolic systole, the extremal length systole has not been studied much so far. 

For flat tori, we will see that the extremal length systole agrees with the systolic ratio, from which it follows that the regular hexagonal torus uniquely maximizes the extremal length systole in genus one (c.f. Loewner's torus inequality \cite{Pu}). 

In \cite{MGVP}, the last two authors of the present paper conjectured that the Bolza surface maximizes the extremal length systole in genus two. This surface, which can be obtained as a double branched cover of the regular octahedron branched over the vertices, is the most natural candidate since it maximizes several other invariants in genus two such as the hyperbolic systole \cite{Jenni}, the kissing number \cite{Schmutz:kiss}, and the number of automorphisms \cite[Section 3.2]{KarcherWeber}. The maximizer of the systolic ratio among all non-positively curved surfaces of genus two is also in the conformal class of the Bolza surface \cite{KatzSabourau} and the same is true for the maximizer of the first positive eigenvalue of the Laplacian among all Riemannian surfaces of genus two with a fixed area \cite{NatayaniShoda}. 

Here we make partial progress toward the aforementioned  conjecture, by showing that the Bolza surface is a strict local maximizer for the extremal length systole.

\begin{thm} \label{thm:main}
The extremal length systole of Riemann surfaces of genus two attains a strict local maximum at the Bolza surface, where it takes the value $\sqrt{2}$.
\end{thm}

Once the curves with minimal extremal length have been identified, the proof that the Bolza surface is a strict local maximizer boils down to a calculation. Indeed, there is a sufficient criterion for generalized systoles to attain a strict local maximum at a point in terms of the derivatives of the lengths of the shortest curves \cite[Proposition~2.1]{Bav:97}, generalizing Voronoi's characterization of extreme lattices in terms of eutaxy and perfection.

The crux of the proof is thus to identify the curves with minimal extremal length. This is not a trivial task because extremal length is hard to compute exactly in general, although it is fairly easy to estimate. In this particular case, we are able to calculate the extremal length of certain curves by finding branched coverings from the Bolza surface to rectangular pillowcases, where extremal length is expressed in terms of elliptic integrals. We then show that all other curves are longer by using a piecewise Euclidean metric to estimate their extremal length. The last step is to compute the first derivative of the extremal length of each shortest curve as we deform the complex structure. These derivatives are encoded by the associated quadratic differentials thanks to Gardiner's formula \cite{G84:MinimalNormProperty}.

The value of $\sqrt{2}$ for the extremal length systole of the Bolza surface came as a surprise to us. Indeed, we initially expressed it as the ratio of elliptic integrals
\[
\left. \int_0^1 \frac{x+1+\sqrt{2}}{ \sqrt{x(1-x^4)}} \, dx  \middle/ \int_1^\infty \frac{x+1+\sqrt{2}}{\sqrt{x(x^4-1)} } \, dx \right.
\]
and numerical calculations suggested that it coincided with the square root of two, which we proved. We later found that this follows from a pair of identities between elliptic integrals called the \emph{Landen transformations}. The flip side is that we appear to have discovered a new geometric proof of these identities.

The other surprising phenomenon is that the curves with minimal extremal length on the Bolza surface correspond to the second shortest curves on the punctured octahedron rather than the first. What is going on is that extremal length is not preserved under double branched covers; it is either multiplied or divided by two depending on the type of curve. While there are twelve shortest curves on the Bolza surface, there are only four on the punctured octahedron. In particular, the punctured octahedron is not perfect. Thus, either the punctured octahedron is not a local maximizer or Voronoi's criterion fails for the extremal length systole. We think the second option is more likely.

To conclude this introduction, we note that the proof that the Bolza surface maximizes the hyperbolic systole in genus two \cite{Jenni} (see also \cite{Bav:hyper}) rests on two ingredients: the fact that every genus two surface is hyperelliptic, and a bound of B\"or\"oczky on the density of sphere packings in the hyperbolic plane. A similar approach is used in \cite{KatzSabourau} to determine the optimal systolic ratio among locally $\CAT(0)$ metrics. While we use the first ingredient in the proof of \thmref{thm:main}, the second ingredient is not available because the extremal length systole is calculated using a different metric for each closed curve. Schmutz's proof that the Bolza surface is the unique local maximizer for the hyperbolic systole \cite[Theorem 5.3]{Schmutz:localmax} is similarly very geometric and does not seem applicable for extremal length. New ideas would therefore be required to remove the word ``local'' from the statement of \thmref{thm:main}.

\subsection*{Organization}
We define extremal length and illustrate how to compute it using Beurling's criterion in \secref{sec:EL}. In \secref{sec:simple}, we show that with the exception of the thrice-punctured sphere, the extremal length systole is only achieved by simple closed curves. \secref{sec:branched} explains how extremal length behaves under branched coverings. In \secref{sec:explicit}, we use elliptic integrals to compute the extremal length of various curves on the punctured octahedron. We then prove lower bounds on the extremal length of all other curves in \secref{sec:estimates} to determine the extremal length systole of the punctured octahedron and the Bolza surface. We prove that the Bolza surface is a strict local maximizer in \secref{sec:derivatives}. Our geometric proof of the Landen transformations is given in Appendix \ref{sec:Landen}, and Appendix \ref{sec:prisms} contains upper bounds for the extremal length systole of six-times-punctured prisms and antiprisms.

\begin{acknowledgements}
The authors would like to thank Misha Kapovich, Chris Leininger, and Dylan Thurston for useful discussions related to this work. Franco Vargas Pallete's research was supported by NSF grant DMS-2001997. Part of this material is also based upon work supported by the National Science Foundation under Grant No. DMS-1928930 while Franco Vargas Pallete participated in a program hosted by the Mathematical Sciences Research Institute in Berkeley, California, during the Fall 2020 semester.
\end{acknowledgements}

%% file: extremal_length.tex
\section{Extremal length} \label{sec:EL}

\subsection{Extremal length} A \emph{conformal metric} on a Riemann surface $X$ is a Borel-measurable map $\rho : TX \to \R_{\geq 0}$ such that $\rho(\lambda v)=|\lambda| \rho(v)$ for every $v \in TX$ and every $\lambda \in \C$. This gives a choice of scale at every point in $X$, with respect to which we can measure length or area. We denote the set of conformal metrics of finite positive area on $X$ by $\conf(X)$.

Given a conformal metric $\rho$ and a map $\gamma$ from a $1$-manifold $I$ to $X$, we define 
\[
\length_\rho(\gamma) := \int_\gamma \rho = \int_I \rho(\gamma'(t)) \, dt
\] if $\gamma$ is locally rectifiable and $\length_\rho(\gamma)=\infty$ otherwise. If $\Gamma$ is a set of maps from $1$-manifolds to $X$, then we set $\ell_\rho(\Gamma) := \inf\{\length_\rho(\gamma) : \gamma \in \Gamma  \}$. Finally, the \emph{extremal length} of $\Gamma$ is
\[
\EL(\Gamma) := \EL(\Gamma,X) := \sup_{\rho \in \conf(X)} \frac{\ell_\rho(\Gamma)^2}{\area_\rho(X)}.
\]

This powerful conformal invariant was introduced by Ahlfors and Beurling in \cite{AhlforsBeurling}. The standard reference on this topic is \cite[Chapter 4]{AhlforsConf}.

Typically, one takes $\Gamma$ to be the homotopy class $[\gamma]$ of a map $\gamma$ from a $1$-manifold to $X$. In this case, we will often abuse notation and write $\EL(\gamma)$ or $\EL(\gamma, X)$ instead of $\EL([\gamma])$ or $\EL([\gamma],X)$. Similarly, we may write $\ell_\rho(\gamma)$ instead of $\ell_\rho([\gamma])$.

\subsection{The extremal length systole}

A \emph{closed curve} in a Riemann surface $X$ is the continuous image of a circle. It is \emph{simple} if it is embedded, and \emph{essential} if it cannot be homotoped to a point or into an arbitrarily small neighborhood of a puncture. The sets of homotopy classes of essential closed curves and of essential simple closed curves in $X$ will be denoted by $\calC(X)$ and $\calS(X)$ respectively.

\begin{defn} \label{def:ELsys}
The \emph{extremal length systole} of a Riemann surface $X$ is defined as
\[
\ELsys(X) :=\inf_{c \in \calC(X)} \EL(c,X).
\]
\end{defn}

The reason for restricting to \emph{essential} closed curves is that the extremal length of any inessential closed curve is equal to zero. We will see in \secref{sec:simple} that if $X$ is different from the thrice-punctured sphere, then we can replace $\calC(X)$ by $\calS(X)$ in \defref{def:ELsys} without affecting the resulting value.

\subsection{Beurling's criterion}
For there to be any hope of computing the extremal length systole of a Riemann surface, we should first be able to compute the extremal length of some essential closed curves. 
The definition of extremal length makes it easy to find lower bounds for it: any conformal metric of finite positive area provides a lower bound. To determine its exact value is harder; all known examples use the following criterion \cite[Theorem 4-4]{AhlforsConf} (c.f. \cite[Section 5.5]{Gromov} or \cite{Bav:92}), which encapsulates the \emph{length-area method}.

\begin{thm}[Beurling's criterion]
Let $\Gamma$ be a set of maps from $1$-manifolds to a Riemann surface $X$. Suppose that $\rho \in \conf(X)$ is such that there is a nonempty subset $\Gamma_0 \subseteq \Gamma$ of shortest curves, meaning that $\length_{\rho}(\gamma) = \ell_{\rho}(\Gamma)$ for every $\gamma \in \Gamma_0$, and that the implication 
\begin{equation} \label{eq:fubini}
\int_{\gamma} f \rho  \geq 0\text{ for all }\gamma \in \Gamma_0 \Longrightarrow \int_X f \rho^2 \geq 0
\end{equation}
holds for every measurable function $f$ on $X$. Then $\EL(\Gamma)= \ell_{\rho}(\Gamma)^2/\area_{\rho}(X)$.
\end{thm}

In practice, one often finds a measure $\mu$ on the set $\Gamma_0$ such that
\begin{equation} \label{eq:fubini2}
\int_X f \rho^2 = k   \int_{\Gamma_0}  \left( \int_{\gamma} f \rho \right) d\mu(\gamma)
\end{equation}
for some constant $k>0$, which obviously implies \eqref{eq:fubini}. In this case, we will say that $\Gamma_0$ \emph{sweeps out $X$ evenly}.

A conformal metric $\rho$ such that $\EL(\Gamma)= \ell_{\rho}(\Gamma)^2/\area_{\rho}(X)$ as in Beurling's criterion is said to be \emph{extremal} for $\Gamma$. Extremal metrics always exist in a weak sense \cite[Theorem 12]{Rodin} (c.f. \cite[Theorem 5.6.C']{Gromov}), though we will not use this. By a convexity argument, if an extremal metric exists, then it is unique (in the sense of equality almost everywhere) up to scaling \cite[Theorem 2.2]{JenkinsBook}.

\subsection{Examples}

Here are some examples of extremal metrics.

\begin{ex} \label{ex:simple}
If $X$ is a Riemann surface with a finitely generated fundamental group and $c$ is the homotopy class of an essential simple closed curve in $X$, then the extremal metric for $c$ is equal to $\sqrt{|q|}$ for an integrable holomorphic quadratic differential $q$ all of whose regular horizontal trajectories belong to $c$, and this quadratic differential is unique up to scaling \cite{Jenkins}. In simpler terms, the extremal metric looks like a Euclidean cylinder with some parts of its boundary glued together via isometries. If the metric $\rho = \sqrt{|q|}$ is scaled so that the height of the cylinder is $1$, then $\ell_{\rho}(c)=\area_{\rho}(X)$ so that the extremal length of $c$ is equal to either of these two quantities (the circumference or the area of the cylinder).
\end{ex}

We refer the reader to \cite{Strebel} for background on quadratic differentials. Given the existence of $q$, the fact that $\sqrt{|q|}$ is extremal follows from Beurling's criterion since the horizontal trajectories of $q$ have minimal length in their homotopy class and integration against $|q|$ is the same as iterated integration along the horizontal trajectories, then against the transverse measure. 

The above result is true more generally if $c$ is the homotopy class of a simple multi-curve with weights \cite{Renelt}. The Heights Theorem of Hubbard and Masur further generalizes the existence and uniqueness of $q$ to equivalence classes of measured foliations \cite{HM79:QuadFoliations} (see also \cite{MS84:Heights}) and this can be used to define the extremal length of such things.

For closed curves that are not simple, very little is known about the extremal metric. The investigations in \cite{Calabi,HZ18:SwissCross,NZ19:Isosystolic} suggest that it might have positive curvature in general. Here are three examples of non-simple closed curves on the thrice-punctured sphere for which the extremal metric is flat but does not come from a quadratic differential.

\begin{figure}[htp]
     \centering
     \begin{subfigure}[b]{0.3\textwidth}
         \centering
         \includegraphics[scale=.75]{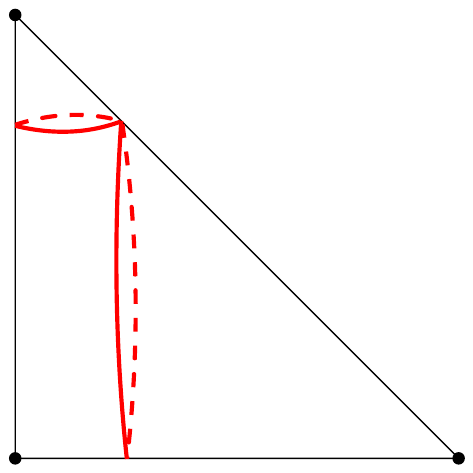}
         \caption{A figure-eight}
         \label{fig:figure-eight}
     \end{subfigure}
     \hfill
     \begin{subfigure}[b]{0.3\textwidth}
         \centering
         \includegraphics[scale=.75]{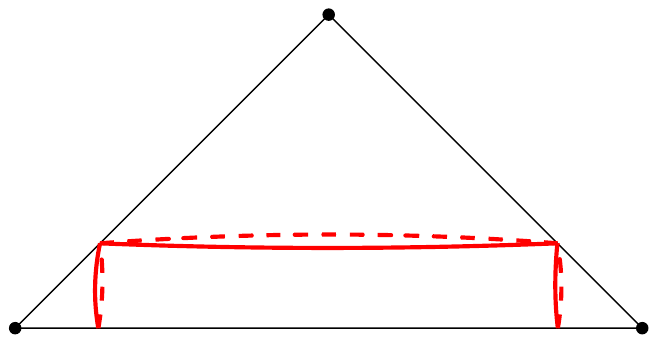}
         \caption{A curve with two self-intersections}
         \label{fig:sausage}
     \end{subfigure}
     \hfill
     \begin{subfigure}[b]{0.3\textwidth}
         \centering
         \includegraphics[scale=.75]{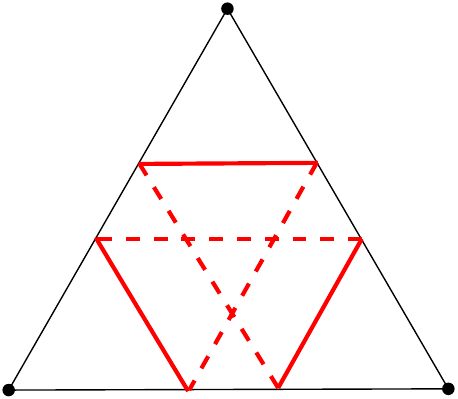}
         \caption{A curve with three self-intersections}
         \label{fig:trefoil}
     \end{subfigure}
     \caption{Some extremal metrics on the thrice-punctured sphere}
     \label{fig:double_triangles}
\end{figure}

\begin{ex} \label{ex:figure-eight}
Take two copies of a Euclidean isoceles right triangle, glue them along their boundary, and puncture the resulting surface at the three vertices. Then the resulting metric $\rho$ is extremal for the figure-eight curve $\gamma$ which winds once around each of the two acute vertices (see \figref{fig:figure-eight}). If the short sides of the triangle have length $1$, then the surface has area $1$ and $\ell_\rho(\gamma)=2$ so that $\EL(\gamma)=4$.
\end{ex}

\begin{ex} \label{ex:sausage}
The metric from \exref{ex:figure-eight} is also extremal for the curve $\gamma$ with two self-intersections depicted in \figref{fig:sausage}. This curve satisfies $\ell_\rho(\gamma)=2\sqrt{2}$ so that $\EL(\gamma)=8$.
\end{ex}

The next example is closely related to \cite[Example 4-2]{AhlforsConf}.

\begin{ex}  \label{ex:isoceles}
Take two copies of a Euclidean equilateral triangle, glue them along their boundary, and puncture the resulting surface at the three vertices. Then the resulting metric $\rho$ is extremal for the curve $\gamma$ depicted in \figref{fig:trefoil}. If the side length of the triangle is equal to $1$, then the area of the surface is equal to $\sqrt{3}/2$ and $\ell_\rho(\gamma)=3$, so that $\EL(\gamma)= 6 \sqrt{3}$.
\end{ex}

In each case, the proof is an application of Beurling's criterion. The first observation is that any closed geodesic in a locally $\CAT(0)$ space (which these surfaces are) minimizes length in its homotopy class. Thus, the closed geodesics depicted in \figref{fig:double_triangles} have minimal length in their homotopy class. Furthermore, in each case the set $\Gamma_0$ of closed geodesics homotopic to $\gamma$ sweeps out the surface evenly. 

To prove this, we use the fact that in each example there is an isometric immersion $\pi$ from an open Euclidean cylinder $C$ to the surface $X$ that sends closed geodesics in $C$ to closed geodesics in the desired homotopy class on $X$. These immersions can be obtained by folding along the dashed lines in \figref{fig:immersions}. We observe that the immersion is $2$-to-$1$ almost every\-where in the first two cases and $3$-to-$1$ almost every\-where in the third case (only the edges of the double triangle get covered fewer times). Therefore, if $f$ is a measurable function on $X$, then
\[
\deg(\pi) \int_X f \rho^2  =  \int_C (f \circ \pi) \,  (\pi^*\rho)^2 =  \int_y \left( \int_{\alpha_y} (f \circ \pi) \, (\pi^*\rho) \right) dy 
 = \int_y \left( \int_{\pi(\alpha_y)} f \rho \right) dy
\]
where $\alpha_y$ is the closed geodesic at height $y$ in $C$, which shows that \eqnref{eq:fubini2} holds.

\begin{figure}[htp]
     \centering
     \begin{subfigure}[b]{0.45\textwidth}
         \centering
         \includegraphics[scale=0.5]{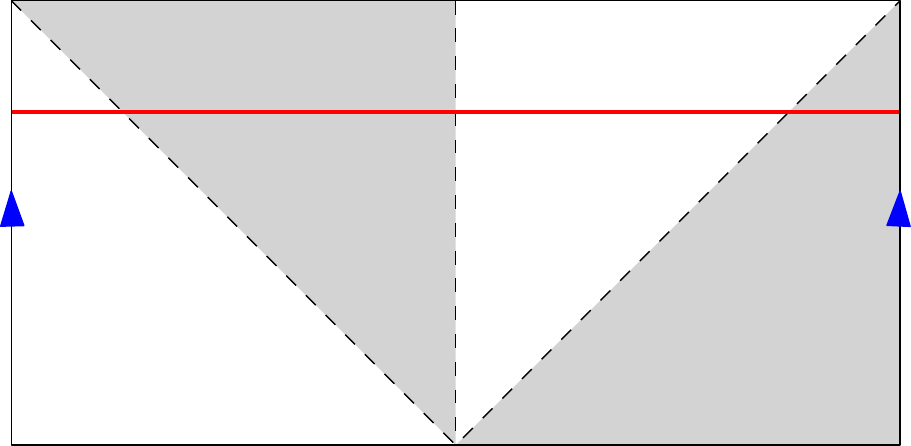}
         \caption{A $2$-to-$1$ immersion}
         \label{fig:figure-eight_cyl}
     \end{subfigure}
     \begin{subfigure}[b]{0.45\textwidth}
         \centering
         \includegraphics[scale=0.473]{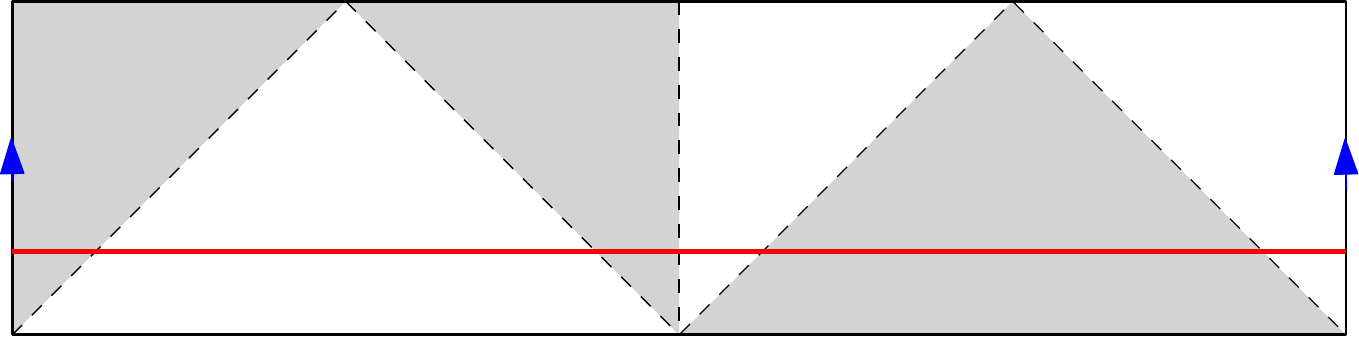}
         \caption{Another $2$-to-$1$ immersion}
         \label{fig:sausage-cyl}
     \end{subfigure}
     \par\bigskip
     \begin{subfigure}[b]{0.6\textwidth}
         \centering
         \includegraphics[scale=0.5]{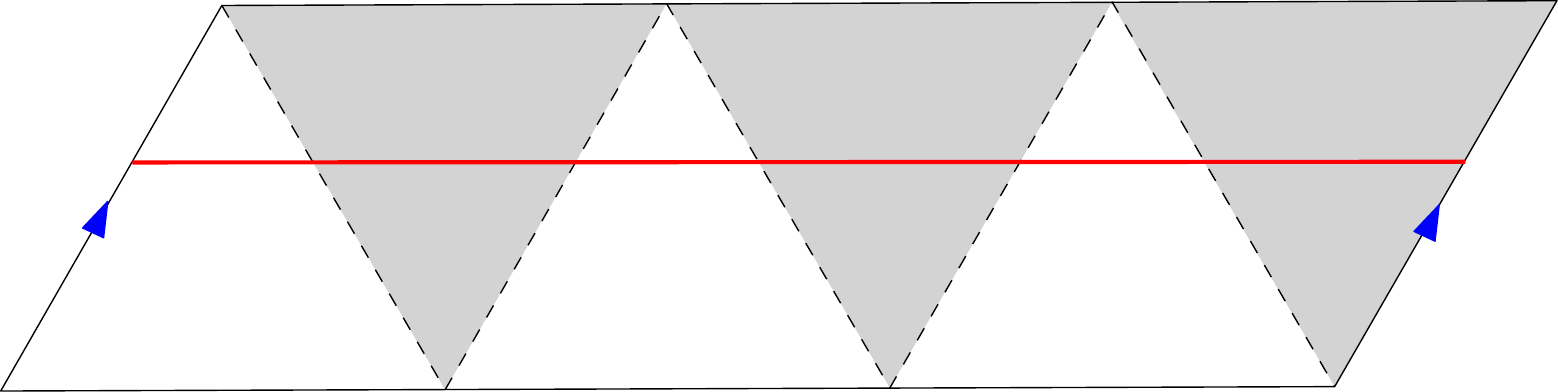}
         \caption{A $3$-to-$1$ immersion}
         \label{fig:trefoil-cyl}
     \end{subfigure}
     \caption{Isometric immersions from cylinders to double triangles}
     \label{fig:immersions}
\end{figure}

\subsection{Pulling curves tight}

In order to apply Beurling's criterion, or even to obtain a lower bound for extremal length given a metric $\rho$, the first step is to determine the infimum $\ell_\rho(\Gamma)$. A useful trick for this is to pull curves tight. This is a straightforward procedure if the surface is compact, but it leads to slight complications in the presence of punctures.

\begin{prop} \label{prop:tight}
Let $\overline{X}$ be a closed Riemann surface equipped with a conformal metric $\rho$ whose induced distance is compatible with the topology on $\overline{X}$, let $X = \overline{X} \setminus P$ where $P \subset \overline{X}$ is a finite set, and let $c$ be the homotopy class of an essential closed curve in $X$. Then there is a closed curve $\gamma$ in $\overline{X}$ such that $\gamma$ is the limit of a sequence $(\gamma_n)$ of curves in $c$, the restriction $\gamma \cap X$ is locally geodesic, and $\length_\rho(\gamma)= \ell_\rho(c)$. If $c$ is the homotopy class of an essential simple closed curve, then the approximating curves $\gamma_n$ can be chosen to be simple. Furthermore, if $\rho$ is piecewise Euclidean with finitely many cone points, then $\gamma$ can be chosen to pass through either a cone point in $X$ or a point in $P$, unless $X$ is a flat torus.
\end{prop}

This is a well-know fact at least for metrics coming from quadratic differentials (see \cite[Chapter V]{Strebel}, \cite[Section 4]{Minsky} or \cite[Section 2.4]{DLR}), but we could not track down a proof in the case where $X$ has punctures. Here we will apply the lemma to polyhedra punctured at the vertices. It is worth pointing out that even when $c$ contains simple closed curves, the length-minimizer $\gamma$ need not be simple; it can have tangential self-intersections.

\begin{proof}
Take any sequence $(\gamma_n) \subset c$ of curves parametrized proportionally to arc length such that $\length_\rho(\gamma_n)$ tends to the infimum $\ell_\rho(c)$ as $n \to \infty$. Since these are uniformly Lipschitz maps from the circle into $\overline X$, which is compact, we can apply the Arzel\`a--Ascoli theorem to extract a subsequence that converges uniformly to a closed curve $\gamma$ in $\overline X$. We also have $\length_\rho(\gamma) \leq \ell_\rho(c)$ since $\gamma$ is $\ell_\rho(c)$-Lipschitz (i.e., length is lower semi-continuous under uniform convergence). Note that $\gamma$ is not reduced to a point since $c$ is essential. 

If $\gamma\cap X$ is not locally geodesic, then we can shorten $\gamma_n$ by a definite amount whenever $n$ is large enough while staying in the same homotopy class, contradicting the hypothesis that $\length_\rho(\gamma_n) \to \ell_\rho(c)$ as $n \to \infty$.

To prove the reverse inequality $\ell_\rho(c) \leq \length_\rho(\gamma)$, the idea is to reconstruct a sequence of curves $\alpha_n$ in $c$ from $\gamma$. We can choose $\alpha_n$ to follow $\gamma$ except where $\gamma$ hits the set $P$. At each of these occurrences, we take $\alpha_n$ to stop a little bit before hitting the given point $p \in P$, wind around $p$ a certain number of times along a small circle centered at $p$, then continue along $\gamma$ where $\gamma$ crosses that circle a second time. Once we have fixed a small enough radius at each of the points in $P$, there is a unique way to choose the winding around each puncture so that $\alpha_n$ belongs to $c$. Indeed, the intersection of a small neighborhood of $\gamma$ with $X$ deformation retracts onto the union of circles and segments of $\gamma$ where we allow $\alpha_n$ to travel. By letting the radii of the circles tend to $0$ as $n \to \infty$, we obtain that $\length_\rho(\alpha_n) \to \length_\rho(\gamma)$ as $n \to \infty$ since the amount of winding around each puncture stays fixed but the circumference of each circle tends to zero. 

Suppose that $c$ contains simple closed curves. We can assume that $\gamma_n$ only has transverse self-intersections for otherwise we can perturb it so that this holds, without increasing its length by more than any positive amount we choose. Let $k$ be the number of self-intersections of $\gamma_n$. If $\gamma_n$ is not simple, then it must bound a monogon or a bigon in $X$ \cite[Theorem 2.7]{HassScott}. By erasing the monogon or by pushing the bigon off to its shorter side, we obtain a curve $\gamma_n'$ which has fewer self-intersections and is not longer than $\gamma_n$ by more than any positive amount we choose, say $1/(kn)$. After a finite number of steps, we obtain a simple closed curve $\beta_n$ in $c$ which is not longer than $\gamma_n$ by more than $1/n$. It follows that $\lim_{n \to \infty} \length_\rho(\beta_n) = \ell_\rho(c)$. We could therefore have chosen $(\beta_n)$ from the start (perhaps ending up with a different limit $\gamma$ after passing to a subsequence).

For the last part of the proof, we assume that $\rho$ is piecewise Euclidean with cone points. Let $Q$ be the set cone points in $X$ and suppose that $\gamma$ is contained in $X\setminus Q$. In particular, $\gamma$ is in $X$ so that it is a closed geodesic by the second paragraph of this proof. The fact that $X\setminus Q$ is locally Euclidean and orientable allows us to push $\gamma$ parallel to itself by using the geodesic flow in the normal direction. Let $\gamma_t$ be the geodesic obtained after pushing $\gamma$ by distance $t\in \R$ to the left (where negative $t$ means pushing to the right). This is well-defined if $t$ is close enough to $0$, but if we bump into $Q$ or $P$, then it is not possible to continue. Let $T$ be the supremum of the set of $s>0$ such that $\gamma_t$ is defined for all $t \in (-s,s)$. 

Suppose that $T < \infty$. By the same argument as in the first paragraph, $\gamma_{t_n}$ subconverges to limiting curves $\gamma_{\pm T}$ in $\overline{X}$ as $t_n \to \pm T$. At least one of these two curves must pass though $Q$ or $P$, otherwise the flow could be continued and $\gamma_t$ would be defined on a larger interval. Thus, we can replace $\gamma$ by one of $\gamma_T$ or $\gamma_{-T}$.

If $T = \infty$, then there is a local isometry from $S^1 \times \R$ to $X \setminus Q$. Since the domain is complete and the range is connected, this local isometry is a covering map \cite[Proposition I.3.28]{BH:MetricSpaces}. But the cylinder $S^1 \times \R$ only covers itself, tori, or Klein bottles. The only one of these which is orientable and whose completion is compact is the torus.
\end{proof}

\subsection{Systolic ratio}

Besides the homotopy class of an essential closed curve in a Riemann surface $X$, there is another set of curves $\Gamma$ whose extremal length one might want to compute, namely, the set $\Gamma_{\mathrm{all}}$ of \emph{all} essential closed curves in $X$.

Given a conformal metric $\rho$ on $X$, the \emph{systole} of $(X,\rho)$ is
\[
\sys(X,\rho) := \ell_\rho(\Gamma_{\mathrm{all}}) = \inf_{\gamma \in \Gamma_{\mathrm{all}}} \length_\rho(\gamma)
\]
and
\[
\SR(X,\rho)  := \frac{\sys(X,\rho)^2}{\area_\rho(X)}
\]
is its \emph{systolic ratio}. By definition, the extremal length of $\Gamma_{\mathrm{all}}$ is equal to
\[
\EL(\Gamma_{\mathrm{all}},X) = \sup_{\rho \in \conf(X)} \frac{\ell_\rho(\Gamma_{\mathrm{all}})^2}{\area_\rho(X)} = \sup_{\rho \in \conf(X)} \SR(X,\rho)=:\overline{\SR}(X),
\]
that is, to the \emph{optimal systolic ratio} in the conformal class of $X$. The \emph{isosystolic pro\-blem} consists in maximizing the optimal systolic ratio over all conformal classes on a given manifold. 

Extremal metrics for the isosystolic problem are known for the torus \cite{Pu}, the projective plane \cite{Pu} and the Klein bottle \cite{Bav:Klein1,Bav:Klein2}. When the conformal class is fixed, there are two further examples of optimal metrics known in genus three \cite[Section 7]{Calabi} and a one-parameter family in genus five \cite[Section 6]{WZ}.

We emphasize that the extremal length systole
\[
\ELsys(X) = \inf_{c \in \calC(X)} \EL(c,X) = \inf_{c \in \calC(X)} \sup_{\rho \in \conf(X)}  \frac{\ell_\rho(c)^2}{\area_\rho(X)}
\]
is different from the optimal systolic ratio
\[
\overline{\SR}(X) = \sup_{\rho \in \conf(X)} \frac{\ell_\rho(\Gamma_\mathrm{all})^2}{\area_\rho(X)}=\sup_{\rho \in \conf(X)} \inf_{c \in \calC(X)} \frac{\ell_\rho(c)^2}{\area_\rho(X)}.
\]
 
 The maximin-minimax principle (which says that $\sup_x \inf_y F(x,y) \leq \inf_y \sup_x F(x,y)$ for any function $F$) yields the inequality
\begin{equation} \label{eq:SRvsSysEL}
\overline{\SR}(X) \leq \ELsys(X).
\end{equation}

If  $X$ is a torus, then equality holds in \eqref{eq:SRvsSysEL}. This is because the extremal metric (for extremal length) is the same for all homotopy classes of curves. Indeed, every essential closed curve $\gamma$ in a torus is homotopic to a power $\alpha^k$ of a simple closed curve $\alpha$. By \exref{ex:simple}, the extremal metric for $\alpha$ (and hence $\gamma$) is realized by a holomorphic quadratic differential $q$ on $X$. The resulting metric is just the flat metric $\rho$ on $X$ because $q$ does not have any singularities (the space of holomorphic quadratic differentials on $X$ is $1$-dimensional). It follows that any homotopy class $c$ with minimal $\rho$-length realizes the extremal length systole, giving \[\overline{\SR}(X) \geq \SR(X,\rho) =\frac{\ell_\rho(c)^2}{\area_\rho(X)}= \EL(c,X)\geq\ELsys(X),\]
and hence $\overline{\SR}(X) = \ELsys(X)$.

By Loewner's torus inequality \cite{Pu}, $\overline{\SR}(X)$ is strictly maximized at the regular hexa\-go\-nal torus, where it takes the value $2/\sqrt{3}$. In fact, it is easy to see that this is the only local maximum (see below). Thus, the same holds for the extremal length systole.

\begin{cor} \label{cor:torus}
The extremal length systole of tori attains a unique (strict) local maximum at the regular hexagonal torus, where it takes the value $2/\sqrt{3}$. 
\end{cor}

The proof of Loewner's torus inequality is quite straightforward once we know that the optimal metric is flat. Indeed, the moduli space of flat tori up to similarity is equal to the modular surface $\H/\PSL(2,\Z)$. The standard fundamental domain for the action of $\PSL(2,\Z)$ on $\H$ is \[F=\{ z \in \H : |\re z| \leq 1/2 \text{ and } |z|\geq 1 \}.\] For any $\tau \in F$, it is easy to see that the shortest non-zero vectors in the lattice $\Z + \tau \Z$ have length $1$. This means that the systolic ratio of the torus $\C / (\Z + \tau \Z)$ is the reciprocal of its area, or $1/ \im \tau$. This quantity is only locally maximized at the corners $\{e^{\pi i / 3}, e^{2\pi i / 3} \}$ of $F$, both of which represent the regular hexagonal torus.

The equality $\overline{\SR}(X) = \ELsys(X)$ also holds for the projective plane because there is only one homotopy class of primitive essential closed curves in that case.

However, if $X$ is the thrice-punctured sphere, then 
\begin{equation}
\overline{\SR}(X) = 2 \sqrt{3}  < 4 = \ELsys(X).
\end{equation}
We explain the first equality here while the second one will be shown in \corref{cor:ELsysPants}. 

\begin{prop} \label{prop:SRsphere}
The Euclidean metric on the double equilateral triangle punctured at the vertices is optimal for the systolic ratio, giving $\overline{\SR}(X) = 2 \sqrt{3}$ for the thrice-puntured sphere.
\end{prop}
\begin{proof}
This is the metric $\rho$ described in \exref{ex:isoceles}. The shortest curves are not those depicted in \figref{fig:trefoil} though, they are figure-eight curves of length $\sqrt{3}$. 

To prove that every essential curve in $X$ has length at least $\sqrt{3}$, we use \propref{prop:tight} to get a closed curve $\gamma$ into the metric completion $\overline{X}$ (the unpunctured double equilateral triangle) such that $\gamma$ is a limit of curves in $c$, satisfies $\length_\rho(\gamma)\leq \ell_\rho(c)$, and $\gamma$ passes through a vertex $v$ of $\overline{X}$.

The curve $\gamma$ must also intersect the edge opposite to $v$, for otherwise the curves in $c$ close enough to $\gamma$ could be homotoped into a neighborhood of $v$, contradicting the assumption that they are essential. As the distance from $v$ to the opposite edge is $\sqrt{3}/2$, we obtain \[\ell_\rho(c) \geq \length_\rho(\gamma) \geq \sqrt{3}.\]

It remains to show that $X$ is swept out evenly by shortest essential closed curves. This is best seen by observing that there is a covering map from the regular hexagonal torus $Y$ punctured at three points to $X$ (see \figref{fig:torus_cover}). The shortest essential closed curves in $Y$ have length $\sqrt{3}$, and those that do not pass through one of the three punctures project to shortest essential closed curves in $X$. These are organised in three parallel families, each of which foliates $Y$ minus three closed geodesics. By picking any one of these parallel families, we obtain an isometric immersion $\pi$ from a union $U$ of $3$ open Euclidean cylinders to $X$ which is $3$-to-$1$ almost everywhere and maps each simple closed geodesic in $U$ to a shortest essential closed curve in $X$. As before, we obtain
\[
3  \int_X f \rho^2 = \int_U (f\circ \pi) \,(\pi^*\rho)^2 = \int_y  \left( \int_{\alpha_y} (f\circ \pi) \, \pi^*\rho \right) dy =\int_{y}  \left( \int_{\pi(\alpha_y)} f \rho \right) dy
\]
for any measurable function $f$ on $X$, where $y$ is a height coordinate in $U$ and $\alpha_y$ is the closed geodesic at height $y$.

By Beurling's criterion, $\rho$ is extremal for the set of curves $\Gamma_\mathrm{all}$, so that
\[
\overline{\SR}(X) =\EL(\Gamma_\mathrm{all},X) = \frac{\ell_\rho(\Gamma_\mathrm{all})^2}{\area_\rho(X)} = \frac{\sqrt{3}^2}{\sqrt{3}/2}=2\sqrt{3}. \qedhere
\]
\end{proof}

\begin{figure}[htp]
    \centering
    \includegraphics[height=1.5in]{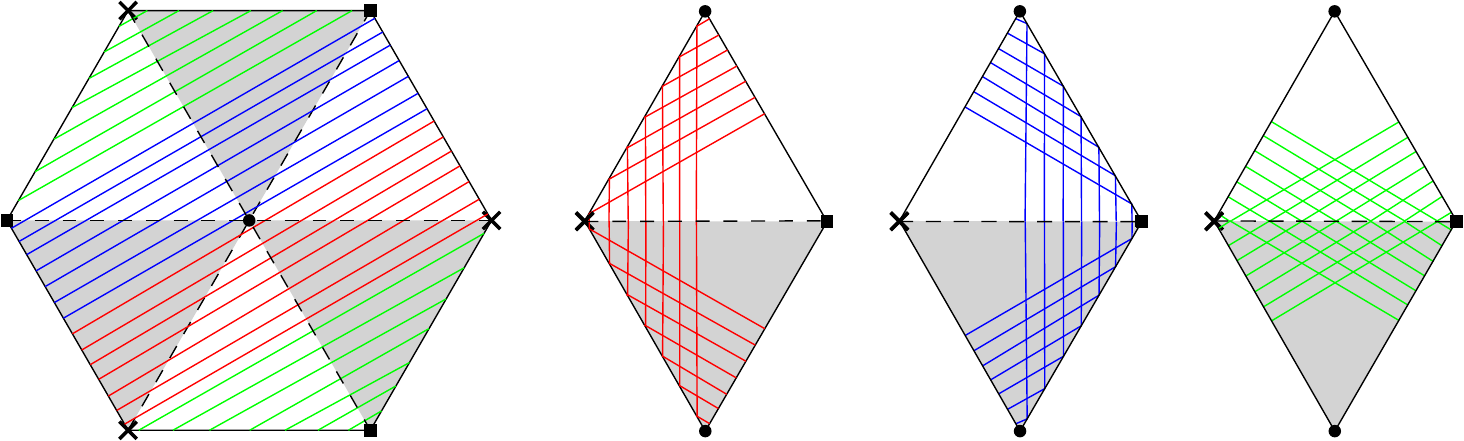}
    \caption{Degree $3$ cover from the hexagonal torus to the double triangle. None of the three homotopy classes of figure-eights sweeps out the double triangle evenly by itself, but together they do.}
    \label{fig:torus_cover}
\end{figure}

Similarly, if $X$ is the Bolza surface, then
\begin{equation}
\overline{\SR}(X) \leq \frac{\pi}{3}  < \sqrt{2} = \ELsys(X)
\end{equation}
where the first inequality is from \cite[Theorem~5.1]{KatzSabourau} and the equality on the right-hand side will be shown in \corref{cor:roottwo}.

%% file: simple.tex
\section{Systoles are simple} \label{sec:simple}

In this section, we show that we can restrict to \emph{simple} closed curves in the definition of the extremal length systole (except in the case where no such curve is essential). For the systole with respect to a fixed metric, this is an easy surgery argument, but since the extremal lengths of different curves are computed using different metrics, the argument is more subtle for the extremal length systole. As the statement is obvious for tori, we restrict to hyperbolic surfaces in this section.

We start by showing that the infimum in the definition of extremal length systole is achieved. This statement will also be used in \secref{sec:derivatives} to show that the extremal length systole is a generalized systole in the sense of Bavard.

\begin{lem} \label{lem:finite}
Let $X$ be a hyperbolic surface of finite area. For any $L>0$, there are at most finitely many homotopy classes $c$ of essential closed curves in $X$ such that $\EL(c,X) \leq L$. In particular, there is a homotopy class $c\in \calC(X)$ such that $\EL(c,X)=\ELsys(X)$.
\end{lem}
\begin{proof}
In the complete hyperbolic metric $\rho$ on $X$, every homotopy class $c \in \calC(X)$ contains a unique closed geodesic and for any $B>0$, there are at most finitely many closed geodesics of length at most $B$ (since this set is discrete and compact). By definition of extremal length, we have $\EL(c,X) \geq \ell_\rho(c)^2 / \area_\rho(X)$. Thus, an upper bound on extremal length implies an upper bound on hyperbolic length, which restricts to finitely many homotopy classes. The infimum is therefore a minimum.
\end{proof}

We proceed to show that the extremal length systole is only realized by essential closed curves with the minimum number of self-intersections possible.

\begin{thm}\label{thm:ELsyssimple}
Let $X$ be a hyperbolic surface of finite area. Then any essential closed curve $\gamma$ in $X$ such that $\EL(\gamma,X)=\ELsys(X)$ is simple unless $X$ is the thrice-punctured sphere, in which case the figure-eight curves are the ones with minimal extremal length. 
\end{thm}
\begin{proof}
Let $\gamma$ be an essential closed curve in $X$ that is not homotopic to a simple closed curve or to a figure-eight on the thrice-punctured sphere. Our goal is to find an essential closed curve $\alpha$ with strictly smaller extremal length than $\gamma$.

There are two ways to perform surgery on $\gamma$ at an essential self-intersection which we call \emph{smoothings}. These are obtained by cutting the circle at two preimages of the intersection point and regluing in the two other possible ways (see \cite[Definition 2.16]{MGT:20}). As such, the smoothings of $\gamma$ have the same image as $\gamma$ and the same length with respect to any metric. One of the two smoothings is a pair of curves and the other one is a single curve.

Suppose first that all three components of the two smoothings of $\gamma$ at some essential self-intersection $x\in \gamma$ are inessential. Consider the smoothing $\gamma'$ of $\gamma$ at $x$ which has two components. By assumption, each component of $\gamma'$ can be homotoped into a puncture of $X$ (it cannot be homotopic to a point since the self-intersection at $x$ is essential). It can therefore be homotoped to a power of a simple loop from $x$ that encloses the puncture. Thus, up to homotopy, we may assume that the image of $\gamma$ is homeomorphic to a figure-eight curve bounding two punctures. Since the other smoothing $\gamma''$ of $\gamma$ at $x$ is also inessential, the third component of $X \setminus \gamma$ must be a punctured disk, so that $X$ is the thrice-punctured sphere. This also implies that the two components of $\gamma'$ are simple, because the fundamental group of $X$ is the free group on the two simple loops $a$ and $b$ forming the figure-eight, and the only words homotopic to the third puncture are conjugate to powers of $ab$. If $\gamma \sim a^m b^n$, then its other smoothing $\gamma''$ is $a^m b^{-n}$, which is conjugate to a power of $ab$ only if $m=\pm 1$ and $n = \mp 1$. We conclude that $\gamma$ is homotopic to a figure-eight on the thrice-punctured sphere, contrary to our assumption.

It follows that at least one of the two smoothings of $\gamma$ has an essential component $\beta$. Then $\EL(\beta) \leq \EL(\gamma)$ since $\ell_{\rho}(\beta) \leq \ell_{\rho}(\gamma)$ for every conformal metric $\rho$ on $X$. This is because any curve $c$ homotopic to $\gamma$ has a smoothing with one component $b$ homotopic to $\beta$ (see \cite[Lemma~2.1]{NC:01} or ~\cite[Lemma~2.17]{MGT:20}) and $b$ is at most as long as $c$ with respect to $\rho$. 

If $\beta$ still has essential self-intersections, then we can repeat the above process of smoothing and keeping an essential component until we are left with a curve $\alpha$ for which this is no longer possible. Then $\alpha$ is either simple or a figure-eight on the thrice-punctured sphere, and we have $\EL(\alpha) \leq \EL(\gamma)$. The tricky part is to prove that the inequality is strict. In either case, we know what the extremal metric $\rho$ for $\alpha$ looks like. Since $\alpha$ was obtained from $\gamma$ after repeated smoothings, we have $\ell_\rho(\alpha) \leq \ell_\rho(\gamma)$. If $\ell_\rho(\alpha) < \ell_\rho(\gamma)$, then $\EL(\alpha)< \EL(\gamma)$ as required. Otherwise, we need to modify the metric $\rho$.

Assume that $\alpha$ is simple. By \exref{ex:simple}, the extremal metric $\rho$ comes from an integrable holomorphic quadratic differential $q$ on $X$. In this metric, the punctures are at a finite distance away, so that the metric completion $\overline{X}$ is a closed surface. By \propref{prop:tight}, the curve $\gamma$ can be pulled tight to a curve $\gamma^*$ in $\overline{X}$ such that $\length_\rho(\gamma^*)=\ell_\rho(\gamma)$. If $\gamma^*$ is a regular closed geodesic in $X$ (that does not pass through cone points), then it must be a power of a simple closed curve (necessarily homotopic to $\alpha$) because its slope with respect to $q$ is constant. As we assumed that $\gamma$ was not simple, we have that $\gamma$ is homotopic to a proper power $\alpha^k$ with $k>1$. We thus have $\ell_\rho(\gamma)= k \ell_\rho(\alpha) > \ell_\rho(\alpha)$ so that $\EL(\alpha)<\EL(\gamma)$. 

On the other hand, if every length-minimizer $\gamma^*$ for the homotopy class $[\gamma]$ passes through a cone point or a puncture, then there are only finitely many possible length-minimizers for $[\gamma]$, as there are only finitely many geodesic segments of length at most $L$ between the points in this finite set (for any $L>0$). In fact, the length-minimizer is unique in this case, but we will not need this. All we need to use is that there is a compact set $K \subset X$ with nonempty interior that is disjoint from every length-minimizer. Then every curve homotopic to $\gamma$ that passes through $K$ is longer than $\ell_\rho(\gamma)$ by a definite amount, for otherwise the compactness argument from the proof of \propref{prop:tight} would yield a length-minimizer passing through $K$. We can therefore decrease $\rho$ by a small amount in $K$ without affecting $\ell_{\rho}(\gamma)$, but thereby decreasing the area of $\rho$. This improved metric shows that $\EL(\gamma) > \EL(\alpha)$, as required.

Next, suppose that $\alpha$ is a figure-eight on the thrice-punctured sphere. Then the extremal metric $\rho$ for $\alpha$ is the double of an isoceles right triangle as described in \exref{ex:figure-eight} (scaled to have edge lengths $1$ and $\sqrt{2}$). If the homotopy class of $\gamma$ does not contain any closed geodesics in $X$, then we can apply the same trick as above to reduce $\rho$ in a set $K$ away from all the length-minimizers to obtain the strict inequality $\EL(\gamma)>\EL(\alpha)$.

The only case left is if $[\gamma]$ contains closed geodesics in $X$. In contrast with the case of simple closed curves, the metric $\rho$ comes from a quartic differential rather than a qua\-dra\-tic differential, so that its closed geodesics can self-intersect (at right angles). However, we can still prove that $\ell_\rho(\gamma)> \ell_\rho(\alpha)$. Suppose on the contrary that $\ell_\rho(\gamma)= \ell_\rho(\alpha)$. Let $\gamma^* \subset X$ be a closed geodesic homotopic to $\gamma$ and let $\gamma^\dagger \subset \overline{X}$ be a curve obtained by pushing $\gamma^*$ to one side until it passes through a puncture, as in the proof of \propref{prop:tight}. Consider the covering map from $\R^2 \setminus \Z^2 \to X$ coming from the regular tiling by of the plane by isoceles right triangles. We can lift $\gamma^\dagger$ under this covering to an arc in the plane with endpoints in $\Z^2$. Since $\gamma^\dagger$ is a limit of closed geodesics, that lift must be a straight line segment $I$. Its length is therefore equal to $\sqrt{m^2+n^2}$ for some integers $m$ and $n$. The only way to obtain $\ell_\rho(\alpha)=2$ is if one of $m$ or $n$ is zero, meaning that $I$ is parallel to one of the coordinate axes. Since any closed geodesic in $X$ elevates to a straight line in $\R^2 \setminus \Z^2$, the deck transformation corresponding to $\gamma^*$ must be a translation. Furthermore, as $\gamma^*$ is parallel to $\gamma^\dagger$ and of the same length, that translation is by distance $2$ along one of the coordinate axes. It follows that $\gamma^*$ is a figure-eight in $X$ (see \figref{fig:figure-eight_cyl} and \figref{fig:figure-eight}). This contradicts our initial hypothesis that $\gamma$ was not homotopic to a figure-eight on the thrice-punctured sphere. We conclude that $\ell_\rho(\gamma)> \ell_\rho(\alpha)$ and hence $\EL(\gamma)> \EL(\alpha)$.
\end{proof}

We obtain the extremal length systole of the thrice-punctured sphere as a bonus.

\begin{cor} \label{cor:ELsysPants}
The extremal length systole of the thrice-punctured sphere is equal to $4$.
\end{cor}
\begin{proof}
\thmref{thm:ELsyssimple} shows that the figure-eight curves have minimal extremal length and \exref{ex:figure-eight} shows that the extremal length of these curves is equal to $4$.
\end{proof}

%% file: hyperelliptic.tex
\section{Branched coverings} \label{sec:branched}

A Riemann surface is \emph{hyperelliptic} if it admits a holomorphic map of degree two onto the Riemann sphere. On a surface of genus $g$, such a holomorphic map has $2g+2$ critical points, called the \emph{Weierstrass points}, and the same number of critical values. The conformal automorphism that swaps the two preimages of any non-critical value and fixes the Weierstrass points is called the \emph{hyperelliptic involution}. 

Every closed Riemann surface of genus two is hyperelliptic, hence arises as a double branched cover of the Riemann sphere branched over six points. As the next lemma shows, extremal length behaves well under branched coverings. This will allow us to reduce computations of extremal length on surfaces of genus two to computations on six-times-punctured spheres.

\begin{lem} \label{lem:branch}
Let $f:X \to Y$ be a holomorphic map of degree $d$ between Riemann surfaces with finitely generated fundamental groups, let $Q \subset Y$ be a finite set containing the critical values of $f$, and let $P \subset f^{-1}(Q)$ be such that $f^{-1}(Q) \setminus P$ is a subset of the critical points of $f$. Then 
\[
\EL(f^{-1}(\gamma) , X \setminus P) = d \cdot \EL(\gamma, Y \setminus Q)
\]
for any simple closed curve $\gamma$ in $Y \setminus Q$.
\end{lem}

Typically, we will take $Q$ to be the set of critical values of $f$ and $P$ to be the empty set, provided that $f^{-1}(Q)$ only consists of critical points.  This is the case if $d=2$ since each point in $Q$ has only one (double) preimage.

\begin{proof}[Proof of \lemref{lem:branch}]
By Jenkins's theorem \cite{Jenkins}, there is a unique integrable holomorphic quadratic differential $q$ on $Y\setminus Q$ whose regular trajectories are all homotopic to $\gamma$ and form a cylinder $C$ of height $1$. With this normalization, the extremal length $\EL(\gamma, Y \setminus Q)$ is equal to the area $\int_{Y \setminus Q} |q|$ of the cylinder.

The pull-back differential $f^* q$ is holomorphic on $X \setminus P$ since simple poles pull-back to regular points or zeros at branch points. Since $f:X \setminus f^{-1}(Q) \to Y \setminus Q$ is a covering map, $f^{-1}(C)$ is a union of cylinders of height $1$, the union of whose core curves is homotopic to $f^{-1}(\gamma)$ relative to $f^{-1}(Q)$, hence relative to $P$ as well. Moreover, $f^{-1}(C)$ contains all the regular horizontal trajectories of $f^* q$.

By Renelt's theorem \cite{Renelt}, the extremal metric for $f^{-1}(\gamma)$ is given by $\sqrt{|f^*q|}$ so that
\[
\EL(f^{-1}(\gamma) , X \setminus P) = \int_{X \setminus P} |f^*q| = d \int_{Y \setminus Q} |q| = d \cdot \EL(\gamma, Y \setminus Q),
\]
as required.
\end{proof}

Note that in the above lemma, the inverse image $f^{-1}(\gamma)$ is not necessarily connected; it may have up to $d$ connected components. It may also happen that some components of $f^{-1}(\gamma)$ are homotopic to each other. In the case where $Y \setminus Q$ is a punctured sphere and $d=2$, the number of components of $f^{-1}(\gamma)$ and whether these components are homotopic to each other is determined by how $\gamma$ separates the punctures. 

\begin{defn}
Let $2\leq m \leq n$ be integers. An \emph{$(m,n)$-curve} on a sphere with $(m+n)$ punctures is a simple closed curve that separates $m$ punctures from $n$ punctures.
\end{defn}

\begin{lem} \label{lem:curves}
Let $f:X \to \CHAT$ be a holomorphic map of degree two from a closed Riemann surface to the Riemann sphere, let $Q \subset \CHAT$ be its set of critical values and let $\gamma \subset \CHAT\setminus Q$ be an $(m,n)$-curve. Then $f^{-1}(\gamma)$ is connected if and only if both $m$ and $n$ are odd, in which case, $f^{-1}(\gamma)$ is separating. If $m$ and $n$ are even, then the two components of $f^{-1}(\gamma)$ are individually non-separating, and they are homotopic to each other if and only if $m=2$.
\end{lem}
\begin{proof}
Recall that $m+n = |Q| = 2g+2$ where $g$ is the genus of $X$, so that $m$ and $n$ have the same parity. The curve $\gamma$ separates the Riemann sphere into a disk $D_m$ with $m$ critical values and a disk $D_n$ with $n$ critical values. The Riemann--Hurwitz formula gives $\chi(f^{-1}(D_m)) = 2 - m$ and $\chi(f^{-1}(D_n)) = 2 - n$. On the other hand, if $f^{-1}(D_m)$ has genus $g_m$ and $b$ boundary components, then $\chi(f^{-1}(D_m)) = 2 - 2g_m - b$. In particular, $b$ is odd if and only $m$ (and $n$) is. Since $b$ is equal to either $1$ or $2$, we have that $f^{-1}(\gamma)$ is connected if and only if $m$ and $n$ are odd. Since every simple closed curve on the sphere is separating, so is its full preimage under $f$.

Assume that $m$ and $n$ are even. Then $f^{-1}(\gamma)$ has two connected components. Similarly, $X \setminus f^{-1}(\gamma)$ has exactly two connected components, namely $f^{-1}(D_m)$ and $f^{-1}(D_n)$. These surfaces are connected because each one is a branched cover of degree two of a disk with non-trivial branching. In particular, given a point in each component of $f^{-1}(\gamma)$, there is a path in $f^{-1}(D_m)$ between them and another such path in $f^{-1}(D_n)$. The concatenation of these two paths gives a closed curve intersecting each component $C$ of $f^{-1}(\gamma)$ only once and transversely, which implies that $C$ is non-separating.

Suppose that the two components of $f^{-1}(\gamma)$ are homotopic to each other. Then these two components bound a cylinder, necessarily equal to one of $f^{-1}(D_m)$ or $f^{-1}(D_n)$. As the Euler characteristic of a cylinder is equal to $0$, one of $m$ or $n$ must be equal to $2$. Since we assumed that $2\leq m\leq n$, we conclude that $m=2$. Conversely, if $m=2$, then $f^{-1}(D_m)$ must be homeomorphic to a cylinder since it has two boundary components and Euler characteristic zero. This implies that the two components of $f^{-1}(\gamma)$ are homotopic to each other.
\end{proof}

Before stating the consequences of the above lemmas for surfaces of genus two, let us see what they mean for surfaces of genus one. Every torus $X$ is hyperelliptic (or rather \emph{elliptic}). The elliptic involution has $4$ critical points and critical values. Let $f:X \to \CHAT$ be the quotient by the elliptic involution and let $Q$ be its set of critical values. Since any essential simple closed curve $\gamma$ in $\CHAT \setminus Q$ is a $(2,2)$-curve, \lemref{lem:curves} tells us that its preimage $f^{-1}(\gamma)$ has two components, both of which are homotopic to a given curve $\alpha \subset X$. Conversely, every essential simple closed curve $\alpha$ in $X$ can be homotoped off of $P$, after which it projects to some $(2,2)$-curve $\gamma$ in $\CHAT \setminus Q$.

 By \lemref{lem:branch} applied with $P = \varnothing$, we obtain
\[
2^2 \EL( \alpha , X)= \EL(2 \alpha , X) = \EL(f^{-1}(\gamma) , X) = 2 \EL(\gamma, \CHAT \setminus Q),
\]
or $\EL(\gamma, \CHAT \setminus Q)=2\EL( \alpha , X)$. By $2 \alpha$ we mean $2$ copies of the curve $\alpha$, and the first equality holds because $\ell_\rho(2 \alpha) = 2 \ell_\rho(\alpha)$ for every conformal metric $\rho$.

By \thmref{thm:ELsyssimple}, the extremal length systole is always achieved by simple closed curves. We conclude that the extremal length systole of a four-times-punctured sphere is equal to twice the extremal length systole of the elliptic double cover branched over the four punctures. Since the quotient of the regular hexagonal torus by the elliptic involution is isometric to the regular tetrahedron, \corref{cor:torus} implies the following.

\begin{cor} \label{cor:tetrahedron}
The extremal length systole of four-times-punctured spheres attains a unique (strict) local maximum at the regular tetrahedron punctured at its vertices, where it takes the value $4/\sqrt{3}$.
\end{cor}

For surfaces of genus two, it is still true that every homotopy class of simple closed curve is preserved by the hyperelliptic involution \cite{HaasSusskind}. However, the situation is a bit more complicated as there are two types of curves. The extremal length of an essential simple closed curve on a surface of genus two is equal to either twice or half the extremal length of some curve on the sphere punctured at the images of the Weierstrass points, depending on the type of curve.

\begin{prop} \label{prop:genus2}
Let $X$ be closed Riemann surface of genus two, let $f:X \to \CHAT$ be a holomorphic map of degree two, let $Q$ be the set of critical values of $f$, and let $\alpha$ be an essential simple closed curve in $X$. Then $\alpha$ is separating if and only if it is homotopic to $f^{-1}(\beta)$ for some $(3,3)$-curve $\beta$ on $\CHAT \setminus Q$, in which case $\EL(\alpha,X) = 2\EL(\beta,Y)$. Similarly, $\alpha$ is non-separating if and only if it is homotopic to either component of $f^{-1}(\beta)$ for some $(2,4)$-curve $\beta$ on $\CHAT \setminus Q$, in which case $\EL(\alpha,X) = \EL(\beta,Y)/2$.
\end{prop}

\begin{proof}
Let $J : X \to X$ be the hyperelliptic involution and let $\alpha^*$ be the geodesic representative of $\alpha$ with respect to the hyperbolic metric on $X$. 

If $\alpha$ is separating, then each component of $X \setminus \alpha^*$ is a one-holed torus preserved by $J$, because $J$ preserves $\alpha^*$ together with its orientation. If $C$ is one such component, then $f(C)$ is a disk and the Riemann-Hurwitz formula tells us that
\[
-1 = \chi(C) = 2 - |f^{-1}(Q) \cap C|
\]
so that $C$ contains $3$ critical points of $f$ and hence $f(C)$ contains $3$ critical values. This shows that $f(\alpha^*)$ is a $(3,3)$-curve. More precisely, $f(\alpha^*)=\beta^2$ for some $(3,3)$-curve $\beta$ since $\alpha^*$ covers its image by degree two. As $\alpha$ is homotopic to $\alpha^*$, we have
\[
\EL(\alpha,X)=\EL(\alpha^*,X)=\EL(f^{-1}(\beta),X)= 2\EL(\beta, \CHAT \setminus Q)
\]
according to \lemref{lem:branch}. 

If $\alpha$ is non-separating, then the isometry $J$ sends $\alpha^*$ to itself in an orientation-reversing manner, so that $\alpha^*$ passes through two Weierstrass points. Let $\eps>0$ be small enough so that the $\eps$-neighborhood of $\alpha^*$ in $X$ is an annulus $A$ that contains only these two Weierstrass points. Then $J$ maps $A$ to itself and exchanges its two boundary components. Let $\beta$ be the image of either boundary component by $f$. Since $f(A)=A / J$ is a disk containing two critical values, $\beta$ is a $(2,4)$-curve on $\CHAT\setminus Q$. Furthermore, $f^{-1}(\beta)=\partial A$ is a union of two curves homotopic to $\alpha$. By \lemref{lem:branch}, we have
\[
2^2 \EL(\alpha, X)=\EL(2 \alpha, X) = \EL(f^{-1}(\beta), X) = 2 \EL(\beta, \CHAT \setminus Q).
\]

The converse statements follow from \lemref{lem:curves}, which tells us that the preimage of a $(3,3)$-curve is connected and separating, while the preimage of a $(2,4)$-curve has two homotopic non-separating components.
\end{proof}

It is perhaps more intuitive to think in terms of embedded cylinders. The inverse image of a $(3,3)$-cylinder in $\CHAT \setminus Q$ has twice the circumference and the same height, while the inverse image of a $(2,4)$-cylinder $C$ consists in two parallel copies of $C$, so the circumference stays the same but the total height is multiplied by two. 

%% file: explicit.tex
\section{From the octahedron to pillowcases} \label{sec:explicit}

The Bolza surface $\calB$ can be defined as the one-point compactification of the algebraic curve
\[
\left\{(x,y)\in \C^2 : y^2 = x(x^4-1)\right\}.
\]
In these coordinates, the hyperelliptic involution takes the form $(x,y)\mapsto (x,-y)$ and the corresponding quotient map is realized by the projection $(x,y) \mapsto x$, which has critical values $\{0, \pm1, \pm i, \infty\}$. By \propref{prop:genus2}, calculating extremal lengths on $\calB$ is equivalent to calculating extremal lengths on $\calO := \CHAT \setminus \{0, \pm1, \pm i, \infty\}$, where $\CHAT$ is the Riemann sphere. This surface is conformally equivalent to the unit sphere $S^2$ in $\R^3$ punctured where the coordinate axes intersect it. Under the stereographic projection, the three great circles obtained by intersecting a coordinate plane in $\R^3$ with $S^2$ map to the two coordinate axes and the unit circle in $\C$. We will refer to the vertices, edges, and faces of this cell division below.

For certain simple closed curves in $\calO$, we are able to explicitly compute their extremal length. We do this by finding branched covers from $\calO$ to four-times-punctured spheres, where extremal length is calculated using elliptic integrals. We will then prove lower bounds for the extremal length of other curves in the next section by using the Euclidean metric on the regular octahedron.

\subsection{The curves}

We distinguish four kinds of curves in $\calO$:
\begin{itemize}
\item A \emph{baseball curve} is a simple closed curve that separates a pair of consecutive edges (adjacent edges that do not belong to a common face) from another such pair.  There are six baseball curves. 

\item An \emph{edge curve} is a simple closed curve that separates an edge from the five edges that are not adjacent to it. There are twelve edge curves, one for each edge.
    
\item An \emph{altitude curve} is a simple closed curve that surrounds a pair of altitudes sharing a common foot. There are twelve altitude curves, one dual to each edge.
    
\item A \emph{face curve} is a simple closed curve that separates two opposite faces. There are four face curves, one for each pair of opposite faces.
\end{itemize}

An example of each is depicted in \figref{fig:curves}. The baseball curves are named like so because of the resemblance with the stitching pattern of a baseball. From this point onward, we will confound closed curves with their homotopy classes.

\begin{figure}[htp]
     \centering
     \begin{subfigure}[b]{0.24\textwidth}
         \centering
         \includegraphics[width=\textwidth]{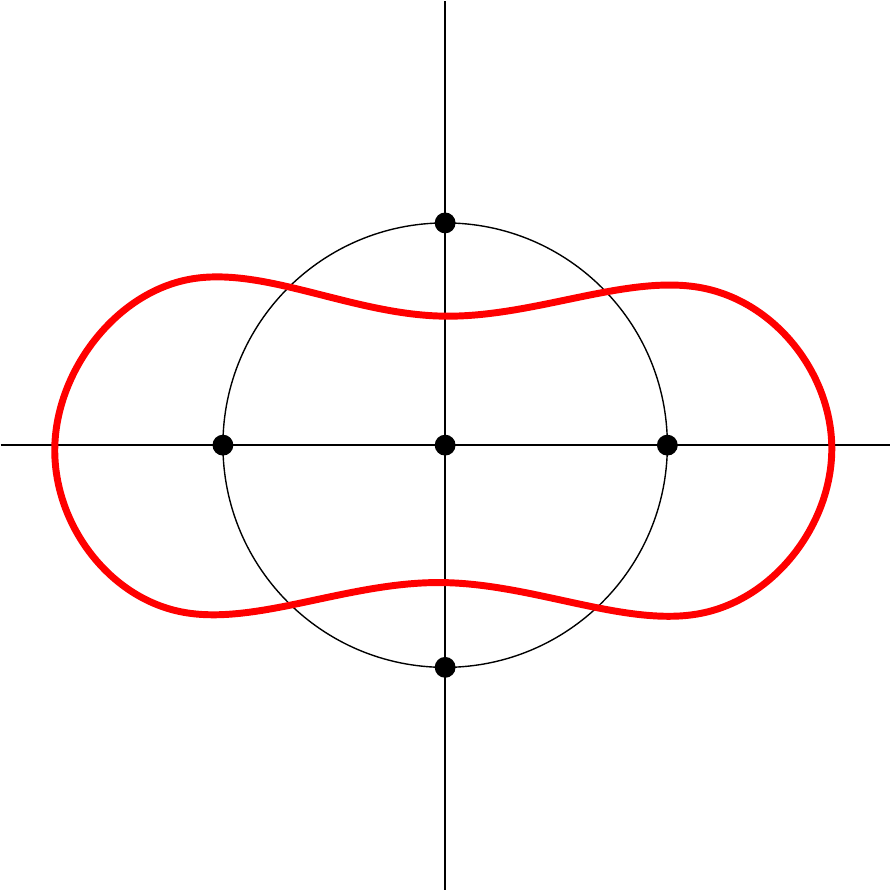}
         \caption{A baseball curve}
         \label{fig:baseball}
     \end{subfigure}
     \begin{subfigure}[b]{0.24\textwidth}
         \centering
         \includegraphics[width=\textwidth]{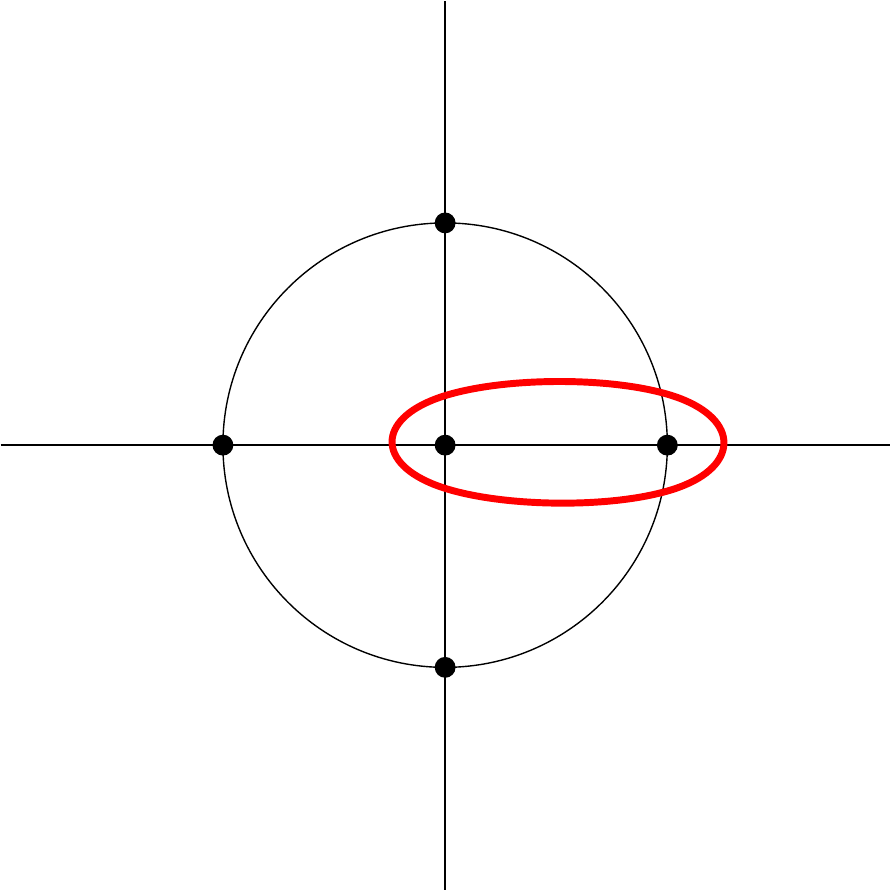}
         \caption{An edge curve}
         \label{fig:edge}
     \end{subfigure}
     \begin{subfigure}[b]{0.24\textwidth}
         \centering
         \includegraphics[width=\textwidth]{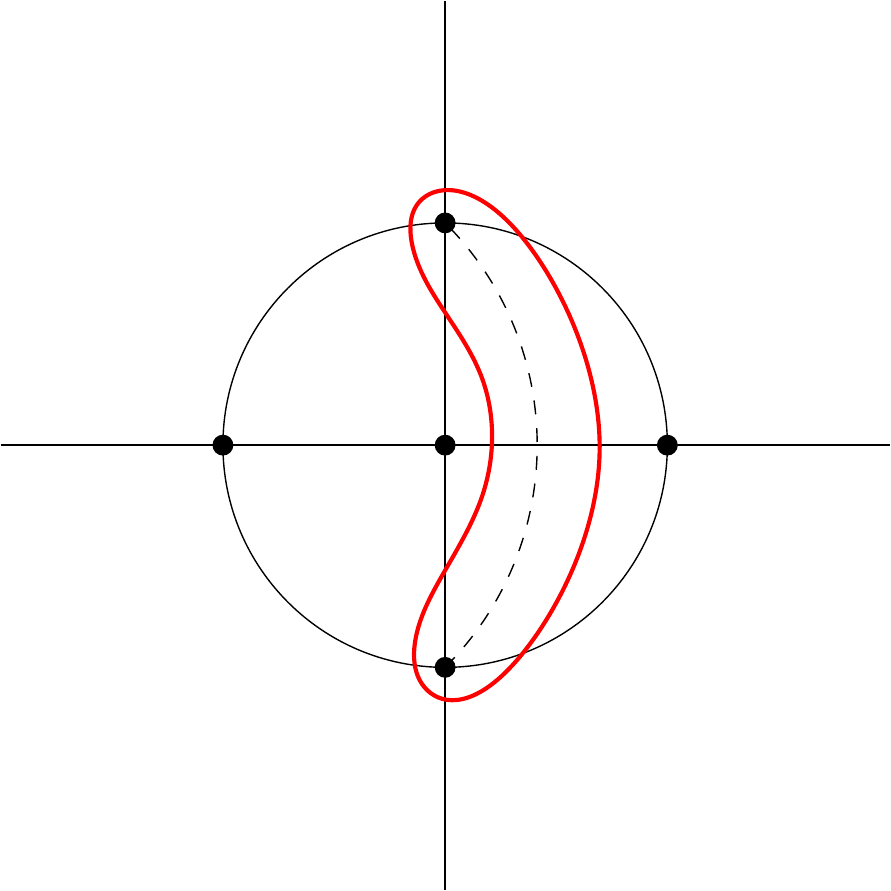}
         \caption{An altitude curve}
         \label{fig:altitude}
     \end{subfigure}
     \begin{subfigure}[b]{0.24\textwidth}
         \centering
         \includegraphics[width=\textwidth]{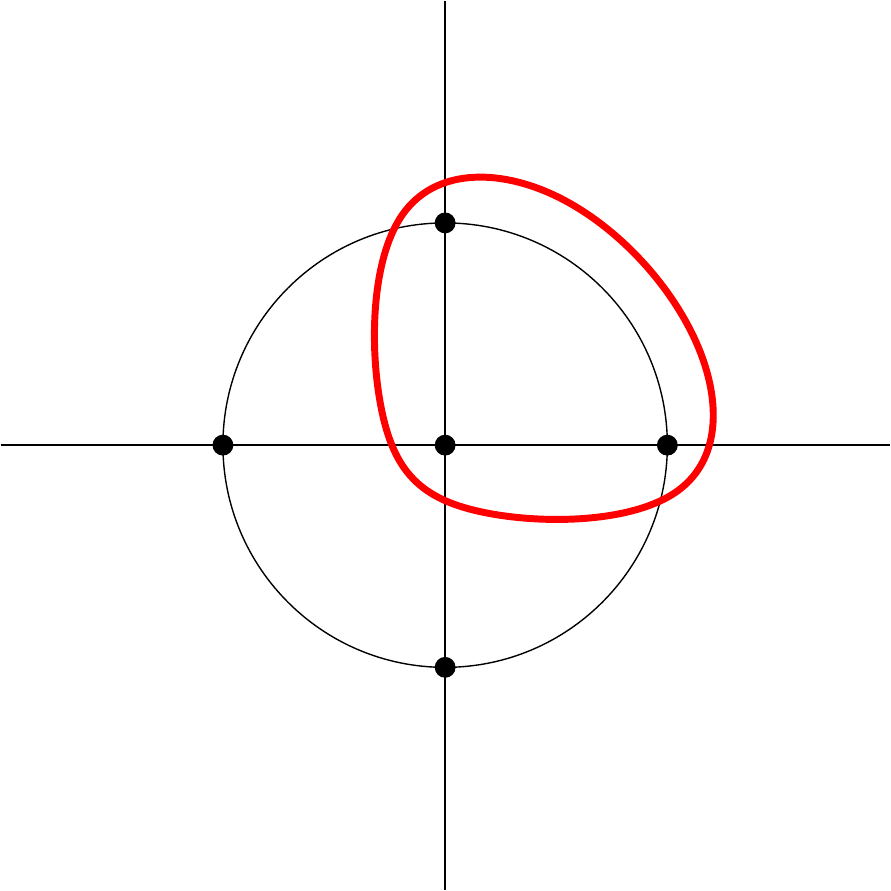}
         \caption{A face curve}
         \label{fig:face}
     \end{subfigure}

     \caption{Some curves on the punctured octahedron}
     \label{fig:curves}
\end{figure}

For any two curves of one kind, there is a conformal automorphism of $\calO$ sending one to the other. Thus, all curves of a given kind have the same extremal length. In each case, there is also a non-trivial conformal automorphism $h$ that preserves the curve. If we quotient $\calO$ by the group generated by $h$, we obtain a holomorphic branched cover $f$ onto the Riemann sphere. By construction, $f$ sends the curve we started with to a power of a simple closed curve in $\CHAT$ punctured at the critical values of $f$ and the images of the vertices of $\calO$.

We find the holomorphic map $f$ for each kind of curve in the next subsections, and use this to compute their extremal length. 

\subsection{The baseball curves}

We begin with the baseball curves, as this case is the simplest. Strictly speaking, we will not need this calculation to determine the extremal length systole of $\calO$ or $\calB$. We only use it to illustrate the method outlined above.

\begin{prop} \label{prop:baseball}
The extremal length of any baseball curve in $\calO$ is equal to $4$.
\end{prop}
\begin{proof}
The baseball curve $\alpha$ depicted in \figref{fig:baseball} is invariant by the rotation $z \mapsto -z$. To quotient by this involution, we apply the squaring map $f(z)=z^2$. This defines covering map $f: \calO \to \CHAT \setminus \{-1,0,1,\infty\}$ of degree two that sends $\gamma$ to $\beta^2$, where $\beta$ is a simple closed curve separating $[0,1]$ from $[-\infty,-1]$.

It is easy to see that $\CHAT \setminus \{-1,0,1,\infty\}$ is conformally equivalent to a square pillowcase, that is, to the double of a Euclidean square punctured at the vertices. Indeed, the closed upper half-plane with vertices at $\{-1,0,1,\infty\}$ is conformally equivalent to a square because it has four-fold symmetry with respect to the point $i$ (the conformal map is given by the Schwarz--Christoffel formula). Doubling across the boundary gives the desired result.

Let $\rho$ be the Euclidean metric on $W=\CHAT \setminus \{-1,0,1,\infty\}$ coming from the square pillowcase of side length $1$. Then $\rho$ is extremal for $\beta$. Indeed, the closed geodesics that go across the squares parallel to the sides have minimal length in their homotopy class because the metric is locally $\CAT(0)$ (this can also be shown using \propref{prop:tight}). The double square is clearly swept out evenly by these closed geodesics homotopic to $\beta$ (this is just Fubini integration on each square). By Beurling's criterion, $\rho$ is extremal, so that
\[
\EL(\beta,W) = \frac{\ell_\rho(\beta)^2}{\area_\rho(W)} = \frac{2^2}{2} = 2
\]
and hence $\EL(\alpha,\calO) = \EL(f^{-1}(\beta),\calO) = 2 \EL(\beta,W) = 4$ by \lemref{lem:branch}.
\end{proof}

Here we got lucky because $\CHAT \setminus \{-1,0,1,\infty\}$ is particularly symmetric. For the other kinds of curves, we will follow a similar approach of finding branched covers onto four-times-punctured spheres, but these will be rectangular rather than square. Their flat metric can be calculated using elliptic integrals, which we briefly discuss now.

\subsection{Elliptic integrals} \label{subsec:elliptic}

For $k \in (0,1)$, the \emph{complete elliptic integral of the first kind} is defined as
\[
K(k) = \int_0^1 \frac{dt}{\sqrt{(1-t^2)(1-k^2 t^2)}}.
\]
The variable $k$ is called the \emph{modulus}. The \emph{complementary modulus} is $k' = \sqrt{1-k^2}$ and the \emph{complementary integral} is $K'(k):=K(k')$. This terminology can be explained by \eqnref{eq:complement}, taken from \cite[p.501]{WW}, in the proof below.

\begin{lem} \label{lem:elliptic}
For any $k \in (0,1)$, the extremal length of the simple closed curve separating the interval $(-1,1)$ from $\pm 1/k$ in $\CHAT \setminus \{ \pm 1, \pm 1/k \}$ is equal to $4K(k)/K'(k)$.
\end{lem}
\begin{proof}
The Schwarz-Christoffel transformation
\[
\zeta \mapsto \int_{\zeta_0}^\zeta \frac{dz}{\sqrt{(1-z^2)(1-k^2z^2)}}
\]
sends the closed upper half-plane to a rectangle $R(k)$ with sides parallel to the coordinate axes, maps $\{\pm 1, \pm 1/k\}$ to the vertices, and is conformal in the interior \cite[p.238--240]{AhlforsComplex}.

The width of $R(k)$ is equal to
\[
W(k) = \int_{-1}^1 \frac{dt}{\sqrt{(1-t^2)(1-k^2 t^2)}} = 2\int_{0}^1 \frac{dt}{\sqrt{(1-t^2)(1-k^2 t^2)}} = 2 K(k)
\]
and its height is equal to
\[
H(k) = \int_{1}^{1/k} \frac{dt}{\sqrt{(t^2-1)(1-k^2 t^2)}}.
\]
The change of variable $t = \sqrt{1-(k')^2s^2}/k$ or equivalently $s= \sqrt{1-k^2t^2}/k'$ shows that
\begin{equation} \label{eq:complement}
\int_{1}^{1/k} \frac{dt}{\sqrt{(t^2-1)(1-k^2 t^2)}} = \int_0^1 \frac{ds}{\sqrt{(1-s^2)(1-(k')^2s^2)}} = K(k')=K'(k).
\end{equation}

Let $P(k)$ be the pillowcase obtained by doubling $R(k)$ across its boundary and puncturing at the vertices. Then $P(k)$ is foliated by closed horizontal geodesics of length $2W(k)=4K(k)$ and its height is equal to $H(k)=K'(k)$. By Beurling's criterion (or \exref{ex:simple}), the flat metric on $P(k)$ is extremal for the homotopy class of these curves. The extremal length is therefore equal to $2W(k)/H(k) = 4 K(k)/K'(k)$.
\end{proof}

Note that the elliptic double cover of $\CHAT \setminus \{ \pm 1, \pm 1/k \}$ is a torus with periods $4K(k)$ and $i 2K'(k)$. For this reason, the integrals $K(k)$ and $K'(k)$ are often called quarter- or half-periods.

Elliptic integrals satisfy several identities that can be used to compute them efficiently. We will require the \emph{upward Landen transformation}
\begin{equation} \label{eq:upward}
K(k) = \frac{1}{1+k} \, K \left( \frac{2\sqrt{k}}{1+k}\right)
\end{equation}
 and the \emph{downward Landen transformation}
\begin{equation} \label{eq:downward}
K'(k) = \frac{2}{1+k} \, K \left( \frac{1-k}{1+k}\right),
\end{equation}
which are valid for every $k \in (0,1)$ \cite[Theorem 1.2]{AGM}. If we define $k^* := 2 \sqrt{k} / (1+k)$, then it is elementary to check that $(k^*)' = (1-k)/(1+k)$. Upon dividing the two Landen transformations, we thus obtain the \emph{multiplication rule}
\begin{equation} \label{eq:multiplication}
    \frac{K(k)}{K'(k)} = \frac12 \, \frac{K(k^*)}{K'(k^*)},
\end{equation}
which is actually what we are going to use.

\subsection{The edge curves} \label{subsec:edge}

We are now able to compute the extremal length of the edge curves.

\begin{prop} \label{prop:roottwo}
The extremal length of any edge curve in $\calO$ is equal to $2\sqrt{2}$.
\end{prop}
\begin{proof}
We first apply a ``rotation'' of angle $-\pi/4$ around the points $\pm i$ to better display the symmetries of the edge curve in \figref{fig:edge}. This is done with the M\"obius transformation
\[
M(z) = \frac{\cos(\pi/8)z-\sin(\pi/8)}{\sin(\pi/8)z +\cos(\pi/8)}= \frac{\sqrt{2+\sqrt{2}} z  - \sqrt{2-\sqrt{2}}}{\sqrt{2-\sqrt{2}} z + \sqrt{2+\sqrt{2}}},
\]
which fixes $\pm i$ and sends $-1$, $0$, $1$, and $\infty$ to
\[
-(\sqrt{2}+1), \quad - (\sqrt{2}-1), \quad \sqrt{2}-1, \quad \text{and} \quad \sqrt{2}+1
\]
respectively. The transformation $M$ also sends the edge curve surrounding $[0,1]$ to a curve $\alpha$ surrounding the interval $[- (\sqrt{2}-1), \sqrt{2}-1]$ in $Z:=\CHAT \setminus\{\pm i,\pm(\sqrt{2}-1), \pm (\sqrt{2}+1) \}$.

Now that everything is symmetric about the origin, we apply the squaring map $f(z) = z^2$. This sends the punctures to $-1$, $3-2\sqrt{2}$ and $3+2\sqrt{2}$, and has critical values at $0$ and $\infty$. Moreover, it maps $\alpha$ to $\beta^2$ where $\beta$ is a curve surrounding the interval $[0,3-2\sqrt{2}]$.

We will compute the extremal length of $\beta$ in $W:= \CHAT \setminus \{ -1 , 0, 3-2\sqrt{2}, 3+2\sqrt{2} \}$. This turns out to be the same as the extremal length of $\beta$ in $W \setminus \{\infty\}$. Indeed, the quadratic differential $q$ realizing the extremal length of $\beta$ in $W$ has two singular horizontal trajectories: the interval $(0,3-2\sqrt{2})$ and $(-\infty,-1)\cup \{\infty\} \cup(3+2\sqrt{2},+\infty)$. So the regular horizontal trajectories of $q$ are homotopic to $\beta$ whether we puncture at $\infty$ or not.

To express the extremal length of $\beta$ as a ratio of elliptic integrals, we first map the punctures $\{-1,0,3-2\sqrt{2},3+2\sqrt{2}\}$ to $\{\pm 1, \pm{1/k}\}$ for some $k \in (0,1)$ via a M\"obius transformation $T$. We begin by applying $z \mapsto 1/z$ to get the points $\{ -1,3-2\sqrt{2},3+2\sqrt{2} , \infty \} $. We then translate by $1$ and scale by $1/(4-2\sqrt{2})$ to end up with
\[
\left\{ 0,1,\frac{4+2\sqrt{2}}{4-2\sqrt{2}}=(\sqrt{2}+1)^2, \infty \right\}.
\]
After these transformations, the curve $\beta$ separates $[0,1]$ from the other two punctures.

The M\"obius transformation $g$ sending $-1$, $1$, and $-1/k$ to $0$, $1$ and $\infty$ respectively is given by
\[
g(z) = \frac{k+1}{2}\left(\frac{z+1}{kz+1}\right).
\]
For it to send $1/k$ to $(\sqrt{2}+1)^2$, we must have
\[
(\sqrt{2}+1)^2 = g(1/k) = \frac{k+1}{2}\left(\frac{\frac{1}{k}+1}{2}\right) = \left(\frac{1+k}{2\sqrt{k}}\right)^2 = \frac{1}{(k^*)^2},
\]
which is equivalent to
\[
k^*= \frac{1}{\sqrt{2}+1} = \sqrt{2}-1.
\]
Also note that
\[
(k^*)' = \sqrt{ 1 - (k^*)^2 } = \sqrt{1 - (\sqrt{2}-1)^2} = \sqrt{2\sqrt{2}-2}.
\]

By \lemref{lem:elliptic}, the extremal length of $\beta$ is equal to $4K(k)/K'(k)$, and the multiplication rule \eqref{eq:multiplication} equates this with $2K(k^*)/K'(k^*)$.

One computes that
\[
(k^*)^* = \frac{2\sqrt{k^*}}{1+k^*} = \sqrt{2\sqrt{2}-2} = (k^*)'  \quad \text{and hence} \quad ((k^*)^*)' = k^*.
\]
The multiplication rule applied to $k^*$ then yields
\[
\frac{2K(k^*)}{K'(k^*)}=\frac{K((k^*)^*)}{K'((k^*)^*)}=\frac{K'(k^*)}{K(k^*)},
\]
so that $K'(k^*)/K(k^*) = \sqrt{2}$ and $\EL(\beta,W) = \sqrt{2}$.

By \lemref{lem:branch}, the extremal length of the edge curves is then
\[
\EL(\alpha,Z) = \EL(f^{-1}(\beta),Z) = 2 \EL(\beta,W \setminus \{ \infty \}) = 2 \EL(\beta,W) = 2 \sqrt{2}. \qedhere
\]
\end{proof}

\begin{remark}
Given a positive integer $n$, the unique $k_n \in (0,1)$ such that $K'(k_n) / K(k_n) = \sqrt{n}$ is called a \emph{singular modulus} \cite[p.139]{AGM}. In the above proof, we showed the well-known fact that $k_2 = \sqrt{2} -1$. The square pillowcase from \propref{prop:baseball} corresponds to the first singular modulus $k_1 = 1/\sqrt{2}$.
\end{remark}

We will need to refer to the quadratic differentials realizing the extremal length of the edge curves later on to compute their derivative as we deform $\calO$. The quadratic differentials associated to the edge curve surrounding $[0,1]$ in $\calO$ can be recovered by pulling back the quadratic differential
\[
\frac{dz^2}{(1-z^2)(1-k^2z^2)},
\]
which realizes the extremal length on the quotient pillowcase (see \lemref{lem:elliptic}), under the branched cover $T \circ f\circ M$ from the above proof.
The result is equal to
\begin{equation} \label{eq:qdiff}
 q = \frac{(z+1+\sqrt{2})^2}{z(1-z^4)}dz^2
\end{equation}
up to a positive constant. We leave the details of this calculation to the diligent reader. 

\begin{figure}[htp]
     \centering
     \includegraphics[scale=.8]{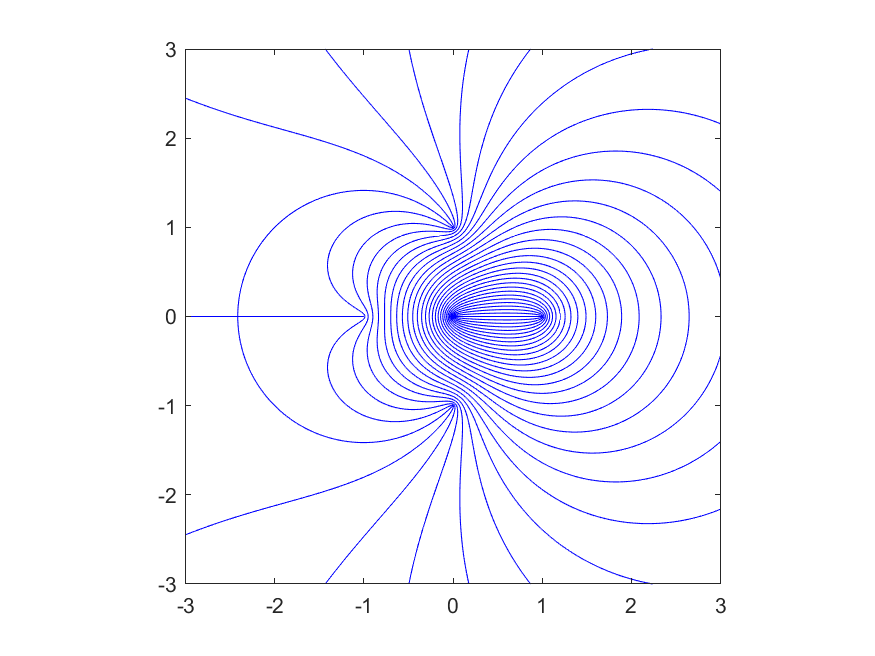}
     \caption{Horizontal trajectories for the quadratic differential realizing the extremal length of an edge curve}
     \label{fig:qd_edge}
\end{figure}

A less computationally intensive approach is to check that $q$ is invariant under the reflection across the real axis as well as the inversion $J$ in the circle of radius $\sqrt{2}$ centered at $-1$. The invariance of $q$ under complex conjugation holds because the coefficients of $q$ are real. Since $J(z)=h(\overline z)$ where $h(z) = (1-z)/(1+z)$, the invariance of $q$ under $J$ follows from its invariance under $h$, which is a calculation:
\begin{align*}
-h^*q &=  \frac{(h(z)+1+\sqrt{2})^2}{h(z)(h(z)^4-1)}(h'(z))^2dz^2 \\
&=  \frac{(h(z)+1+\sqrt{2})^2}{h(z)(h(z)^4-1)}\left( \frac{-2}{(1+z)^2}\right)^2dz^2 \\
&=  \frac{4((1-z)+(1+\sqrt{2})(1+z))^2}{(1+z)(1-z)((1-z)^4-(1+z)^4)}dz^2 \\
&=  \frac{8(z+1+\sqrt{2})^2}{(1-z^2)(-8z-8z^3)}dz^2 \\
& = \frac{(z+1+\sqrt{2})^2}{z(z^4- 1)}dz^2 = -q.
\end{align*}

From these symmetries and by checking the sign of $q$ at certain tangent vectors, it follows that the union of the critical horizontal trajectories of $q$ is equal to
\[
(-\infty,-1] \cup \left\{ -1- \sqrt{2}e^{i \theta} : \theta \in [-3\pi/4, 3\pi/4] \right \} \cup [0,1].
\]
The complement of this locus is homeomorphic to a cylinder, which forces the regular horizontal trajectories of $q$ to be closed and homotopic to each other. By the result of Jenkins cited in \exref{ex:simple}, $q$ is the extremal quadratic differential for the edge curve around $[0,1]$. A plot of the horizontal trajectories of $q$ is shown in \figref{fig:qd_edge}.

\subsection{The altitude curves}

The extremal length of the altitude curves can be calculated similarly as for the edge curves. We will not use this result; we only include it because it is one of the few examples symmetric examples where we can compute the extremal length.

\begin{prop} \label{prop:altitude}
The extremal length of any altitude curve in $\calO$ is equal to \[4K(u) / K(u') \in \left[5.8768721265012 \pm 1.18\cdot 10^{-14}\right],\]
where $u=\frac{\sqrt{2+\sqrt{2}}}{2}$ and $u'=\frac{\sqrt{2-\sqrt{2}}}{2}$.
\end{prop}
\begin{proof}
We use the same model $Z=\CHAT \setminus\{\pm i,\pm(\sqrt{2}-1), \pm (\sqrt{2}+1)\}$ as for the edge curves, and take the altitude curve $\alpha$ to surround the vertical line segment between $i$ and $-i$.

Upon squaring, the punctures map to $-1$, $3-2\sqrt{2}$ and $3+2\sqrt{2}$. However, we also need to puncture at the critical values $0$ and $\infty$. The curve $\alpha$ maps to $\beta^2$ where $\beta$ is a simple closed curve surrounding $[-1,0]$. The critical trajectories for the quadratic differential associated to $\beta$ on $W= \C \setminus \{ -1,0,3-2\sqrt{2} \}$ are equal to $[-1,0]$ and $[3-2\sqrt{2},+\infty]$. Since $3+2\sqrt{2}$ lies along one of these trajectories, this puncture does not affect extremal length.  

We translate by $1$ to map the punctures of $W$ to $0$, $1$, $4-2\sqrt{2}$ and $\infty$. The M\"obius tranformation sending $-1$ to $0$, $1$ to $1$ and $-1/k$ to $\infty$ is equal to
\[
g(z) = \frac{k+1}{2}\left(\frac{z+1}{kz+1}\right).
\]
For it to send $1/k$ to $4-2\sqrt{2}$ we need to have
\[
4-2\sqrt{2} = g(1/k) = \left(\frac{1+k}{2\sqrt{k}}\right)^2 = \frac{1}{(k^*)^2},
\]
which can rewrite as $k^* = \sqrt{\frac{1}{4-2\sqrt{2}}} = \frac{\sqrt{2+\sqrt{2}}}{2}$.
The complementary modulus is
\[
(k^*)' = \sqrt{1-(k^*)^2} = \frac{\sqrt{2-\sqrt{2}}}{2}.
\]

By \lemref{lem:branch}, the above remark about the superfluous puncture, \lemref{lem:elliptic}, and the multiplication rule \eqref{eq:multiplication}, we obtain
\[
\EL(\alpha,Z) = \EL(f^{-1}(\beta),Z) = 2 \EL(\beta, W \setminus \{ 3+2\sqrt{2} \}) = 2\EL(\beta,W) = \frac{8K(k)}{K'(k)} = \frac{4K(k^*)}{K'(k^*)}.
\]

To compute elliptic integrals numerically, computer algebra systems use the arithmetic-geometric mean $M$ via the formula $K(a) = \pi / (2M(1,a'))$ \cite[Theorem 1.1]{AGM}. The ratio of two complementary integral then becomes $K(a)/K'(a) = M(1,a)/M(1,a')$. The C library Arb for arbitrary-precision interval arithmetic developed Fredrik Johansson \cite{Johansson} contains an implementation of $M$, which provides certified error bounds for its calculation. The Arb package is available in SageMath \cite{sagemath}, where the command
\begin{center}
\texttt{4*CBF(sqrt(2+sqrt(2))/2).agm1()/CBF(sqrt(2-sqrt(2))/2).agm1()}
\end{center}
certifies that $4K(k^*)/K'(k^*)$ belongs to the interval $\left[5.8768721265012 \pm 1.18\cdot 10^{-14}\right]$.
\end{proof}

\subsection{The face curves}

We finish with the face curves, which have the smallest extremal length of the lot.

\begin{prop} \label{prop:zigzag}
The extremal length of any face curve in $\calO$ is equal to \[6K(v)/K(v')\in \left[2.79957467136936 \pm 8.4\cdot10^{-15}\right],\] where $v=1/\sqrt{27+15\sqrt{3}}$ and $v' = 1/\sqrt{27-15\sqrt{3}}$.
\end{prop}
\begin{proof}
We start by applying a M\"obius transformation $M$ to display the three-fold symmetry of the face curves. We do this by sending $0$, $1$ and $i$ to the third roots of unity. Writing the explicit formula for $M$ is a bit messy, so we do some trigonometry instead in order to determine where $M$ sends the other vertices of $\calO$. The key observation is that $M$ sends the coordinate axes and the unit circle to three congruent circles that pairwise intersect orthogonally at one of the three third roots of unity.

\begin{figure}[htp]
     \centering
     \includegraphics[width=0.65\textwidth]{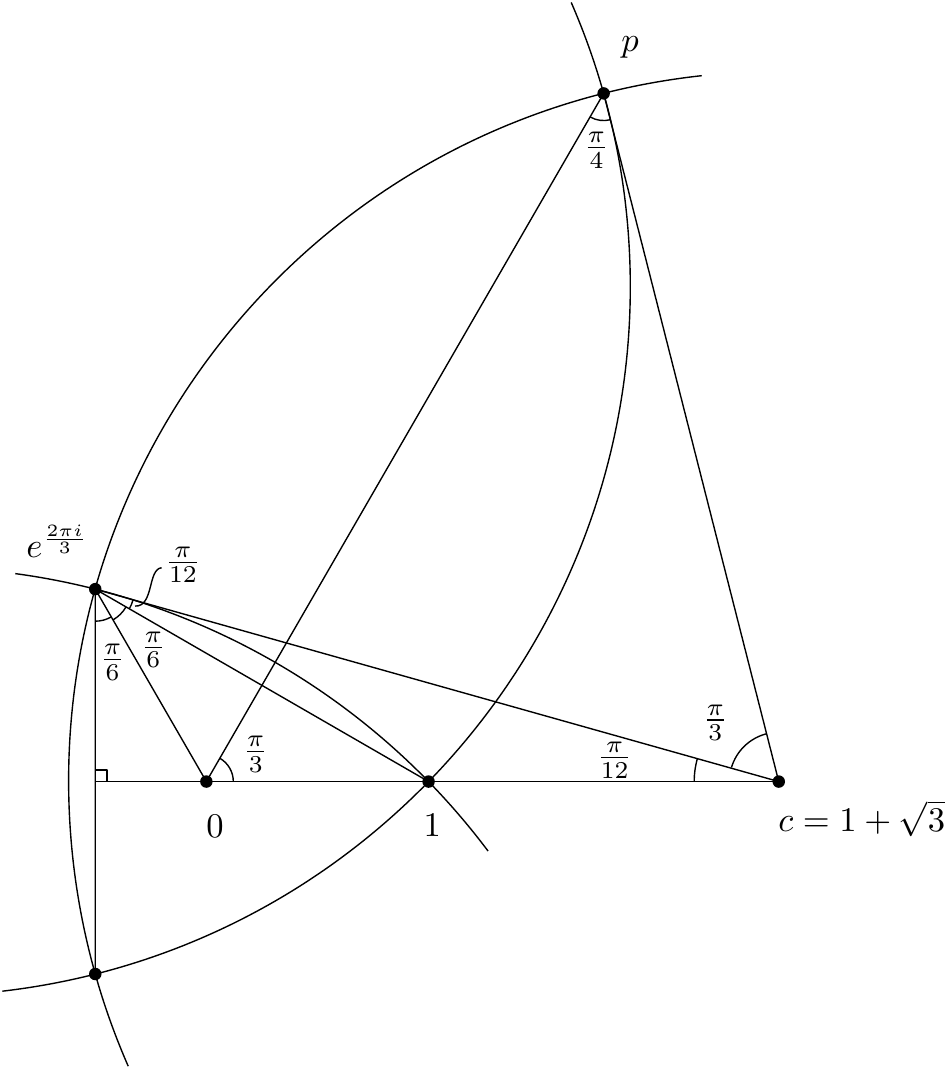}
     \caption{Three congruent circles intersecting at right angles}
     \label{fig:venn}
\end{figure}

Among the centers of these three circles, let $c$ be the one that lies on the real axis. By looking at the angles in \figref{fig:venn}, we see that the triangle with vertices $1$, $e^{2\pi i /3}$ and $c$ is isoceles, so that $c = 1 + \sqrt{3}$. Now consider the triangle $\Delta$ with vertices at $0$,  $c$, and the intersection point $p=r e^{ i \pi/ 3 }$ between two of the circles. The interior angles of $\Delta$ at $0$, $p$, and $c$ are equal to $\pi/3$, $\pi/4$ and $5\pi / 12$ respectively. By the law of sines, we have
\begin{align*}
r & = \frac{(1+\sqrt{3})}{\sin(\frac{\pi}{4})} \sin\left(\frac{5\pi}{12}\right)\\
&= \frac{(1+\sqrt{3})}{\sin(\frac{\pi}{4})} \left( \cos\left(\frac{\pi}{6}\right) \sin\left(\frac{\pi}{4}\right)+\sin\left(\frac{\pi}{6}\right) \cos\left(\frac{\pi}{4}\right)\right) \\
& = \frac{(1+\sqrt{3})^2}{2} = 2 + \sqrt{3}.
\end{align*}
It follows that $Z = M(\calO) = \CHAT \setminus \{ 1,e^{2\pi i / 3},e^{-2\pi i / 3}, -(2 + \sqrt{3}), (2+\sqrt{3})e^{\pi i/ 3}, (2+\sqrt{3})e^{-\pi i/ 3}  \}$. By construction, $M$ sends the face curve depicted in \figref{fig:face} to the homotopy class of the circle $\alpha$ of radius $2$ centered at the origin.

The cubing map $f(z)=z^3$ quotients out the three-fold symmetry. It maps the punctures to $1$ and $-(2+\sqrt{3})^3$ and has critical values at $0$ and $\infty$. We thus let $W = \C \setminus \{-(2+\sqrt{3})^3, 0, 1 \}$. The curve $f(\alpha)$ is equal to $\beta^3$ where $\beta$ separates $[0,1]$ from the other two punctures.

In order to express $\EL(\beta,W)$ as a ratio of elliptic integrals, we apply another M\"obius transformation to send $0$, $1$ and $\infty$ to $-1$, $1$ and $1/k$ respectively, for some $k\in (0,1)$. The inverse of the required transformation is
\[
g(z)= \frac{1-k}{2}\left(\frac{1+z}{1-kz}\right).
\]
Since we want $g$ to map $-1/k$ to $-(2+\sqrt{3})^3$, we obtain the equation
\[
-(2+\sqrt{3})^3 = g(-1/k) = \left(\frac{1-k}{2}\right)\left(\frac{1-\frac{1}{k}}{2}\right)= -\left( \frac{1-k}{2\sqrt{k}} \right)^2,
\]
or
\[
\frac{1-k}{2\sqrt{k}} = (2+\sqrt{3})^{3/2}.
\]
Observe that
\[
\frac{1}{((k^*)')^2} = \left( \frac{1+k}{1-k} \right)^2 = 1 + \left( \frac{2 \sqrt{k}}{1-k} \right)^2 = 1 + \frac{1}{(2+\sqrt{3})^3}=1 + (2-\sqrt{3})^3=27-15\sqrt{3}
\]
so that
\[
(k^*)^2 =  1 - ((k^*)')^2 = 1 - \frac{1}{1+(2-\sqrt{3})^3} = \frac{(2-\sqrt{3})^3}{1+(2-\sqrt{3})^3} = \frac{1}{1+(2+\sqrt{3})^3}=\frac{1}{27+15\sqrt{3}}.
\]

By \lemref{lem:branch}, \lemref{lem:elliptic}, and the multiplication rule \eqref{eq:multiplication}, we obtain
\[
\EL(\alpha, Z)= \EL(f^{-1}(\beta),Z)= 3 \EL(\beta, W) = \frac{12K(k)}{K'(k)} =  \frac{6K(k^*)}{K'(k^*)}.
\]
The command
\begin{center}
\texttt{6*CBF(1/sqrt(27+15*sqrt(3))).agm1()/CBF(1/sqrt(27-15*sqrt(3))).agm1()}
\end{center}
in SageMath certifies that $6K(k^*)/K'(k^*)$ belongs to the interval stated.
\end{proof}

Although we will not use this, we note that the quadratic differential realizing the extremal length of the face curve $\alpha$ surrounding the cube roots of unity in \[\CHAT \setminus \{ 1,e^{2\pi i / 3},e^{-2\pi i / 3}, -(2 + \sqrt{3}), (2+\sqrt{3})e^{\pi i/ 3}, (2+\sqrt{3})e^{-\pi i/ 3}  \}\] is
\[
q=\frac{z}{(1-z^3)(z^3+(2+\sqrt{3})^3)} dz^2.
\]
Indeed, this differential is invariant under rotations by $e^{2\pi i / 3}$ and is positive along \[(-\infty,-2-\sqrt{3}) \cup (0,1).\] This implies that the set of critical horizontal trajectories is a tripod joining the origin to $1$, $e^{2\pi i / 3}$, $e^{-2\pi i / 3}$ together with a ray from each of the other three punctures out to infinity. Since the complement of the set of critical horizontal trajectories is homeomorphic to an annulus, all other horizontal trajectories are closed and homotopic to each other.

By a similar argument, the regular vertical trajectories of $q$ are all homotopic to a simple closed curve $\gamma$ intersecting the face curve $\alpha$ six times. The equality case in Minsky's inequality \cite[Lemma 5.1]{MinskyIneq} then yields \[\EL(\gamma)= 6^2 / \EL(\alpha) = 6K(v')/K(v) \in [12.8590961934912 \pm 6.81\cdot 10^{-14}].\]
Some horizontal and vertical trajectories of $q$ are shown in \figref{fig:qd_face}.

\begin{figure}[htp]
     \centering
     \includegraphics[scale=.8]{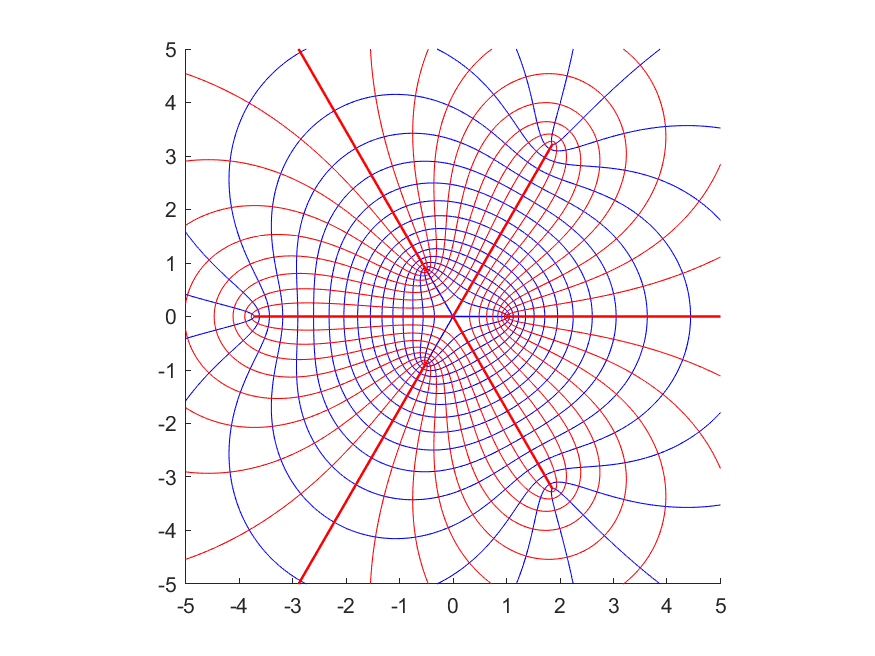}
     \caption{Horizontal (blue) and vertical (red) trajectories for the quadratic differential realizing the extremal length of a face curve}
     \label{fig:qd_face}
\end{figure}

%% file: estimates.tex
\section{Geodesics on the regular octahedron} \label{sec:estimates}

In this section, we prove lower bounds on the extremal length of all simple closed curves on the punctured octahedron $\calO$ other than those for which we could compute it explicitly. We do this by using the conformal metric $\rho$ coming from the surface of the regular octahedron, scaled so that the edges have length $1$ and hence the total area is equal to $2\sqrt{3}$. This metric, which we call the \emph{flat metric}, is in the conformal class of $\calO$. Indeed, any of the curvilinear faces of $\calO$ can be mapped conformally onto an equilateral face of the regular octahedron via the Riemann mapping theorem, and the mapping can be extended to all of $\calO$ by repeated Schwarz reflection across the sides. 

The first step is to obtain lower bounds on the infimal length $\ell_\rho(c)$ for various homotopy classes of simple closed curves in $\calO$.

\begin{lem} \label{lem:geodesic}
Let $c$ be a homotopy class of simple closed curves in $\calO$. If $\ell_\rho(c)=2$, then $c$ is an edge curve. If $\ell_\rho(c)=3$, then $c$ is a face curve. Otherwise, $\ell_\rho(c)\geq 2\sqrt{3}$.
\end{lem}
\begin{proof}
Since $\rho$ is locally Euclidean and its completion $\overline{\calO}$ is compact, \propref{prop:tight} tells us that there is a closed curve $\gamma \subset \overline{\calO}$ which is a limit of a sequence of simple closed curves $\gamma_n$ in $c$ with $\lim_{n\to\infty} \length_\rho(\gamma_n)=\ell_\rho(c)$, passes through at least one vertex, and is geodesic away from the vertices. The curve $\gamma$ is thus a sequence of saddle connections on the regular octahedron. 

Suppose first that $\gamma$ passes through only one vertex of $\overline{\calO}$. Since $\gamma$ is geodesic away from that vertex, if it self-intersects, then the intersection must be transverse. It follows that $\gamma_n$ is not simple if $n$ is large enough, contrary to our assumption. We deduce that $\gamma$ is simple. However, there does not exist a simple geodesic loop from a vertex to itself on the regular octahedron \cite[Theorem~3.1]{Fuchs}. Therefore, $\gamma$ passes through at least two vertices. 

Note that the distance between adjacent vertices is equal to $1$ and the distance between opposite vertices is equal to $\sqrt{3}$. In particular, if $\gamma$ passes through two opposite vertices, then its length is at least $2\sqrt{3}$. We can thus assume that $\gamma$ passes through two or three pairwise adjacent vertices and no other vertices.

If $\gamma$ passes through two adjacent vertices, then its length is at least $2$ with equality only if it traces an edge twice. In that case, $c$ is the homotopy class of the associated edge curve because a small neighborhood of the edge in $\calO$ is a twice punctured disk, and the only essential simple closed curve in such a surface is the boundary curve.

By inspecting the planar development of the regular octahedron, we see that the second shortest geodesic segment between two adjacent vertices has length $\sqrt{7}$ as depicted in \figref{fig:root7} (this is the shortest vector length in the hexagonal lattice after $0$, $1$, $\sqrt{3}$, and $2$). Thus, if $\gamma$ passes through only two adjacent vertices and has length larger than $2$, then it has length at least $1+\sqrt{7} > 2\sqrt{3}$.

\begin{figure}[htp]
     \centering
     \includegraphics[height=1.5in]{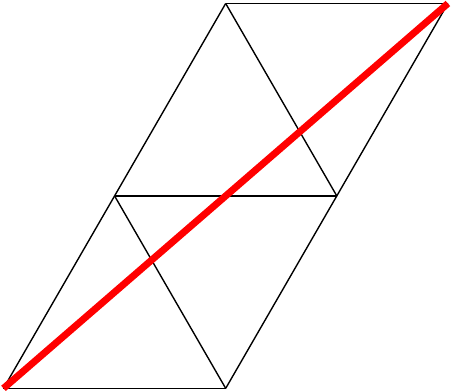}
     \caption{The second shortest saddle connection between adjacent vertices on the octahedron}
     \label{fig:root7}
\end{figure}

The final case to consider is if $\gamma$ passes through three adjacent vertices. Then its length is at least $3$, with equality only if $\gamma$ traces the boundary of a triangle. In this case, $c$ has to be a face curve. Indeed, the approximating curve $\gamma_n$ can bypass each vertex by circling either inside or outside the triangle, but if it passes inside, then $\gamma_n$ can be shortened by a definite amount within $c$, contradicting the hypothesis that $\length_\rho(\gamma_n)$ tends to the infimum $\ell_\rho(c)$.

By the above argument, the next shortest closed curve that passes through only three adjacent vertices and is otherwise geodesic has length at least $1+1+\sqrt{7} > 2\sqrt{3}$. 
\end{proof}

From this, we easily deduce that the essential simple closed curves  in $\calO$ with the first and second smallest extremal length are the face curves and the edge curves.

\begin{cor} \label{cor:spectrum}
The first and second smallest extremal lengths of essential simple closed curves in $\calO$ are $6K(v)/K'(v)$ and $2\sqrt{2}$,
where $v=1/\sqrt{27+15\sqrt{3}}$, and these are realized by the face curves and the edge curves respectively.
\end{cor}
\begin{proof}
The extremal lengths of the face and edge curves were determined in Propositions \ref{prop:zigzag} and \ref{prop:roottwo}. Note that $6K(v)/K'(v) < 2.799575 < 2.828427 < 2\sqrt{2}$.

Let $c$ be any other homotopy class of essential simple closed curve in $\calO$ and let $\rho$ be the flat metric on $\calO$. By \lemref{lem:geodesic}, we have $\ell_\rho(c) \geq 2\sqrt{3}$ so that
\[
\EL(c,\calO)\geq \frac{\ell_\rho(c)}{\area_\rho(\calO)}\geq \frac{(2\sqrt{3})^2}{2\sqrt{3}}= 2\sqrt{3} > 2 \sqrt{2}. \qedhere
\]
\end{proof}

An interesting consequence is that \cite[Proposition 3.3]{FanoniParlier}---stating that if two shortest closest geodesics on a hyperbolic surface intersect twice, then one of them must bound a pair of punctures---is false for the extremal length systole. Indeed, any two distinct face curves intersect twice and they are all $(3,3)$-curves. Given two distinct faces curves $f_1$ and $f_2$, observe that their smoothing with two components is a union of two edge curves $e_1$ and $e_2$. Despite the fact that
\begin{equation} \label{eq:smoothing}
 \ell_\rho(e_1)+\ell_\rho(e_2) =\ell_\rho(e_1 \cup e_2) \leq \ell_\rho(f_1 \cup f_2)  = \ell_\rho(f_1)+\ell_\rho(f_2)
\end{equation}
for every conformal metric $\rho$ and hence  $\EL(e_1 \cup e_2) \leq \EL(f_1 \cup f_2)$ \cite[Lemma 4.16]{MGT:20}, we do not have $\min_j \EL(e_j)\leq \max_k \EL(f_k)$. In other words, the smoothing argument from the proof of \thmref{thm:ELsyssimple} does not work for pairs of shortest closed curves. By \eqnref{eq:smoothing}, the edge curves are shorter than the face curves for any conformal metric $\rho$ invariant under the automorphisms of $\calO$, such as the hyperbolic metric. In the hyperbolic metric on $\calO$, the shortest closed geodesics are the edge curves and furthermore $\calO$ globally maximizes the hyperbolic systole in its moduli space. This is true more generally for principal congruence covers of the modular surface (see \cite[Theorem 13]{Schmutz:congruence} and \cite[Theorem 7.2]{Adams}).

We then apply \corref{cor:spectrum} to determine the extremal length systole of the Bolza surface.

\begin{cor}  \label{cor:roottwo}
The extremal length systole of the Bolza surface $\calB$ is equal to $\sqrt{2}$, and the only simple closed curves with this extremal length are lifts of edge curves in $\calO$.
\end{cor}
\begin{proof}
By \thmref{thm:ELsyssimple}, the extremal length systole is realized by essential simple closed curves, so we may restrict our attention to these. 

Let $f: \calB \to \overline{\calO}$ be the quotient by the hyperelliptic involution. If $\alpha$ is a separating curve in $\calB$, then $\alpha$ is homotopic to one component of $f^{-1}(\beta)$ for some $(2,4)$-curve in $\calO$ and $\EL(\alpha,\calB)=\EL(\beta,\calO)/2$ according to \propref{prop:genus2}. If $\beta$ is an edge curve, then we get $\EL(\alpha,\calB) = \sqrt{2}$. Otherwise, $\EL(\beta,\calO) > 2\sqrt{2}$ by \corref{cor:spectrum}, so that $\EL(\alpha,\calB) > \sqrt{2}$.

If $\alpha$ is a non-separating curve in $\calB$, then \propref{prop:genus2} tells us that $\alpha$ is homotopic to $f^{-1}(\beta)$ for some $(3,3)$-curve $\beta$ in $\calO$ and $\EL(\alpha,\calB)=2 \EL(\beta,\calO)$. By \corref{cor:spectrum}, we have $\EL(\beta,\calO) \geq 6K(v)/K'(v) > 2.799574$ with $v$ as above, so that \[\EL(\alpha,\calB) \geq 12K(v)/K'(v) > 5.599148 > \sqrt{2},\]
as required.
\end{proof}

%% file: derivatives.tex
\section{Derivatives} \label{sec:derivatives}

In this section, we prove that the extremal length systole of genus two surfaces attains a strict local maximum at the Bolza surface $\calB$.

\subsection{Generalized systoles}
Let $M$ be a smooth connected manifold, let $\calI$ be an arbitrary set and for each $\alpha \in \calI$, let $L_\alpha : M \to \R$ be a $C^1$ function. If for every $x \in M$ and $B\in \R$, there is a neighborhood $U$ of $x$ such that the set 
\[
 \{ \alpha \in \calI : L_\alpha(y) \leq B \text{ for some } y \in U  \}
\]
is finite, then the infimum
\[
\mu(x)= \inf_{\alpha \in \calI} L_\alpha(x)
\]
is called a \emph{generalized systole} \cite{Bav:05}. The technical condition is there to ensure that the resulting function $\mu$ is continuous. As mentioned in the introduction, the extremal length systole fits into this framework.

\begin{lem}
The extremal length systole $\ELsys$, as a function on the Teichm\"uller space $\calT(S)$ of a surface $S$ of finite type, is a generalized systole in the sense of Bavard.
\end{lem}
\begin{proof}
First of all, the Teichm\"uller space $\calT(S)$ is a connected complex manifold. 

If $S$ is the thrice-punctured sphere, then $\calT(S)$ is a  point and there is nothing to show except that for every $B>0$, there are only finitely many essential closed curves with extremal length at most $B$, which was proved in \lemref{lem:finite}.

Otherwise, the extremal length systole is only achieved by simple closed curves according to \thmref{thm:ELsyssimple}, so we might as well restrict to these when taking the infimum. The extremal length of such a curve is $C^1$ on $\calT(S)$ \cite[Proposition 4.2]{GardinerMasur}.
 
 The last thing to do is to improve the pointwise finiteness of \lemref{lem:finite} to a local one. To this end, recall that the logarithm of extremal length is Lipschitz with respect to the Teichm\"uller distance $d$. More precisely, we have
\[
\EL(\alpha, X) \leq e^{2d(X,Y)}\EL(\alpha, Y) 
\]
for every $\alpha \in \calS(S)$ and every $X,Y \in \calT(S)$ (see e.g. \cite{Kerckhoff}). This implies the required local finiteness, as an upper bound for $\EL(\alpha, Y)$ in a ball centered at $X$ implies an upper bound on $\EL(\alpha, X)$, hence restricts $\alpha$ to a finite subset of $\calS(S)$.
\end{proof}

This implies that $\sysEL$ is continuous on Teichm\"uller space and therefore on moduli space. Since extremal length and hyperbolic length tend to zero together \cite[Corollary 2]{Maskit}, Mumford's compactness criterion implies that $\sysEL$ attains its maximum.

\subsection{Perfection and eutaxy}
Given $x\in M$, we denote by $\calI_x$ the set of $\alpha \in \calI$ such that $L_\alpha(x)=\mu(x)$. Bavard's definition of eutaxy and perfection \cite[D\'efinition 1.2]{Bav:05} is easily seen to be equivalent to the following, which we find easier to state.

\begin{defn} A point $x \in M$ is \emph{eutactic} if for every tangent vector $v \in T_x M$, the following implication holds: if $dL_\alpha(v)\geq 0$ for all $\alpha \in \calI_x$, then $dL_\alpha(v) = 0$ for all $\alpha \in \calI_x$.
\end{defn}
\begin{defn}
A point $x \in M$ is \emph{perfect} if for every $v \in T_x M$, the following implication holds: if $dL_\alpha(v) = 0$ for all $\alpha \in \calI_x$, then $v = 0$.
\end{defn}

If $x\in M$ is perfect and eutactic, then for every $v \in T_x M \setminus \{ 0 \}$ there is some $\alpha \in \calI_x$ such that $dL_\alpha(v) < 0$, and it follows easily that $\mu$ attains a strict local maximum at $x$ \cite[Proposition 2.1]{Bav:97}. Bavard proved that the converse holds if the functions $L_\alpha$ are convexoidal (i.e., convex up to reparametrization) along the geodesics for a connection on $M$ \cite[Proposition 2.3]{Bav:97}, thereby generalizing a theorem of Voronoi on the systole of flat tori \cite{Voronoi} and its analogue for hyperbolic surfaces \cite{Schmutz:localmax}. Akrout further proved that generalized systoles obtained from convex length functions are topologically Morse, with singularities equal to the eutactic points and index equal to the rank of the linear map $(dL_\alpha)_{\alpha \in \calI_x}$ \cite{Akrout}.

We do not know if there exists a connection on Teichm\"uller space with respect to which the extremal length functions are convexoidal; they are not convexoidal along Teichm\"uller geodesics \cite{nonconvex} or horocycles \cite{horo}. However, all we need here is the easy direction of Bavard's result, namely, that perfection and eutaxy are sufficient to have a local maximum.

\subsection{Triangular surfaces}

A Riemann surface $X$ is \emph{triangular} or \emph{quasiplatonic} if any of the following equi\-va\-lent conditions hold \cite[Theorem 4]{Wolfart}:
\begin{itemize}
    \item the quotient of $X$ by its group of conformal automorphisms is isomorphic to a sphere with three cone points (as an orbifold);
    \item $X \cong \H / \Gamma$ where $\Gamma$ is a normal subgroup of a triangle rotation group;
    \item $X$ is an isolated fixed point of a finite subgroup of the mapping class group acting on Teichm\"uller space.
\end{itemize}

Bavard showed that if the collection of length functions $\{ L_\alpha : \alpha \in \calI \}$ is invariant under a finite group $G$ acting by isometries on $M$, then any isolated fixed point of $G$ is eutactic \cite[Corollaire~1.3]{Bav:05}. Since the set of extremal length functions is invariant under the action of the mapping class group (which acts by isometries on Teichm\"uller space), we conclude that any triangular surface is eutactic for the extremal length systole (c.f. \cite[p.255]{Bav:05}).
 
 \subsection{The Bolza surface}
 
 It is clear that the punctured octahedron $\calO$ is quasiplatonic. The same is true for the Bolza surface $\calB$ since any conformal automorphism of $\calO$ lifts to $\calB$ in two different ways (related by the hyperelliptic involution). In fact, $\calO$ and $\calB$ are \emph{Platonic} in the sense that they admit tilings by regular polygons (triangles in this case) such that their group of conformal and anti-conformal automorphisms acts transitively on the flags of these tiling (triples consisting of a vertex, an edge, and a face, each contained in the next). The surfaces $\calO$ and $\calB$ are therefore eutactic for the extremal length systole.
 
To prove that $\ELsys$ attains a strict local maximum at $\calB$, all we have left to show is that $\calB$ is perfect, which amounts to proving that a certain linear map is injective. We can do this calculation at the level of six-times-punctured spheres. Indeed, the hyperelliptic involution induces a diffeomorphism $\Phi:\calT(S_{2,0}) \to \calT(S_{0,6})$ from the space of closed surfaces of genus two to the space of six-times-punctured spheres. 

Recall from \corref{cor:roottwo} that the curves in $\calB$ with the smallest extremal lengths are lifts of edge curves in $\calO$ (this is the set $\calI_x$ in the notation of generalized systoles). Let $f: \calB \to \overline{\calO}$ be the quotient by the hyperelliptic involution. For ease of notation, we will write $\EL_\alpha(Z)$ instead of $\EL(\alpha,Z)$. If $\beta \subset \calO$ is an edge curve, $\alpha$ is either component of $f^{-1}(\beta)$, and $Z$ is any surface in $\calT(S_{2,0})$, then $\EL_\alpha(Z) = \EL_\beta(\Phi(Z))/2$ according to \propref{prop:genus2}. It follows that $d\EL_\alpha(v) = d\EL_\beta(d\Phi(v))/2$ for every tangent vector $v \in T_\calB \calT(S_{2,0})$. Since $d\Phi$ is a bijection, to prove that $\calB$ is perfect, it thus suffices to show that if $d\EL_\beta(w) = 0$ for every edge curve $\beta \subset \calO$, then $w=0$. 

\subsection{Gardiner's formula}

The tangent space $T_X \calT(S)$ to Teichm\"uller space at a surface $X$ is isomorphic to a quotient of the space of essentially bounded Beltrami differentials (or $(-1,1)$-forms) on $X$, while the cotangent space $T^*_X \calT(S)$ can be identified with the set of integrable holomorphic quadratic differentials (or $(2,0)$-forms) on $X$. We can define a bilinear pairing $T_X \calT(S) \times T^*_X \calT(S) \to \C$ between these objects by sending any pair $(\mu, q)$ to the integral of the $(1,1)$-form $\mu q$ over $X$.

Given an essential simple closed curve $\alpha \subset X$, recall that there is a holomorphic quadratic differential $q_\alpha$ all of whose regular horizontal trajectories are homotopic to $\alpha$, and that $q_\alpha$ is unique up to scaling. Gardiner's formula \cite[Theorem 8]{G84:MinimalNormProperty} says that the logarithmic derivative of the extremal length of $\alpha$ in the direction of $\mu$ is
\[
\frac{d\EL_\alpha(\mu)}{\EL_\alpha(X)} = \frac{2}{\|q_\alpha\|} \re \int_X \mu q_\alpha.
\]
for every $\mu \in T_X \calT(S)$, where $\|q_\alpha\|=\int_X |q_\alpha|$ is the area of the induced conformal metric.

\subsection{Punctured spheres}

The Teichm\"uller space $\calT(S_{0,p})$ of a punctured sphere admits local coordinates to $\C^{p-3}$. Indeed, if we map three of the punctures to $0$, $1$ and $\infty$ with a M\"obius transformation, then the location of the remaining $p-3$ punctures determines the surface locally (i.e., up to the action of the mapping class group). From this point of view, the tangent space $T_X \calT(S_{0,p})$ is naturally isomorphic to $\C^{p-3}$, whereby we attach a complex number $V(z_j)$ to each puncture $z_j$ of $X$ other than $0$, $1$, and $\infty$.

The two points of view can be reconciled by extending $V$ to a smooth vector field on $\CHAT$ that vanishes at $0$, $1$, and $\infty$. The Beltrami form $\mu = \overline{\partial} V$ then represents the same infinitesimal deformation as the one obtained by flowing the punctures along $V$ \cite[Equation (1.5)]{AhlforsRemarks}. Furthermore, the pairing of this deformation with an integrable holomorphic quadratic differential $q$ on $X$ is given by
\[
\int_X \mu q =  - \pi \sum_{j=1}^{p} \Res_{z_j} \left( q \cdot V(z_j)\frac{\partial}{\partial z}\right),
\]
where the sum is taken over all the punctures of $X$ \cite[Lemma 8.2]{FB15:HoloCouchTheorem}. Note that the product of a quadratic differential with a vector field is a $1$-form (locally, $dz^2 \cdot \frac{\partial}{\partial z} = dz$). The residue of a $1$-form $\omega$ at a point $z_j$ is defined in the usual way as $\Res_{z_j}(\omega)=\frac{1}{2\pi i} \oint_\gamma \omega$ where $\gamma$ is a small counterclockwise loop around $z_j$. 

Combined with Gardiner's formula, this yields
\begin{equation} \label{eq:residue}
\frac{d\EL_\alpha(\mu)}{\EL_\alpha(X)} = -\frac{2\pi}{\|q_\alpha\|} \re \sum_{j=1}^{p} \Res_{z_j} \left( q_\alpha \cdot V(z_j)\frac{\partial}{\partial z}\right)
\end{equation}
for any essential simple closed curve $\alpha$ in a punctured sphere $X$ and any $V:\CHAT\setminus X \to \C$.

\subsection{The edge curves}

A real basis $B=\{ b_1, \ldots, b_6 \}$ for the tangent space $T_\calO \calT(S_{0,6})$ is given by the vectors $\frac{\partial}{\partial z}$ and $-i\frac{\partial}{\partial z}$ at each of the three punctures $i$, $-1$, and $-i$. For each edge curve $\beta \in \calE$, to compute the derivatives $d\EL_\beta(b_j)$ we first need to write down a formula for the associated quadratic differential $q_\beta$.

Recall from subsection \ref{subsec:edge} that the quadratic differential for the edge curve $\beta_{0,1}$ surrounding the edge $[0,1]$ is
\[
q = \frac{(z+1+\sqrt{2})^2}{z(1-z^4)}dz^2=\frac{(z+1+\sqrt{2})^2}{z(1-z)(1+z)(i-z)(i+z)}dz^2.
\]
The residue of $q$ in the direction of $\frac{\partial}{\partial z}$ at each of $i$, $-1$, and $-i$ is equal to
\[
-\left(\frac{1+\sqrt{2}}{2}\right)(1+i), \quad -\frac{1}{2} , \quad \text{and} \quad -\left(\frac{1+\sqrt{2}}{2}\right)(1-i)
\]
respectively. Note that $\re \left(\Res_{a} \left( q \cdot -i\frac{\partial}{\partial z}\right)\right) = \im\left( \Res_{a} \left( q \cdot \frac{\partial}{\partial z}\right) \right)$ at any point $a$. Up to a constant, the logarithmic derivative of $\EL_{\beta_{0,1}}$ with respect to the basis $B$ is thus given by
\[
-\frac{\|q\|}{2\pi}\frac{d\EL_{\beta_{0,1}}}{\EL_{\beta_{0,1}}(\calO)}  = \left(\frac{-1-\sqrt{2}}{2}, \frac{-1-\sqrt{2}}{2},-\frac{1}{2}, 0 , \frac{-1-\sqrt{2}}{2}, \frac{1+\sqrt{2}}{2} \right)
\]
according to \eqnref{eq:residue}.

To compute the quadratic differential $q_\beta$ associated to a given edge curve $\beta\in \calE$, it suffices to find a M\"obius transformation $g: \calO \to \calO$ that sends $\beta$ to $\beta_{0,1}$. The desired quadratic differential is then the pullback $g^* q$. The required M\"obius transformations for the $12$ edge curves are $z$, $i z$, $-z$, $-iz$, $-(z-i)/(z+i)$, $-i(z-i)/(z+i)$, $(z-i)/(z+i)$, $i(z-i)/(z+i)$, $1/z$, $i/z$, $-1/z$, and $-i/z$. For each of these, we calculated the pullback differential and the residues at $i$, $-1$, and $-i$ using Maple \cite{Maple} (which we found was better than other computer algebra systems at cancelling factors on the denominator to compute residues) to avoid calculation mistakes. The resulting matrix for $-\frac{\|q_\beta\|}{\pi}\frac{d\EL_{\beta}(b_j)}{\EL_{\beta}(\calO)}$, where $\beta$ ranges over the edge curves and $b_j$ ranges over the basis vectors, is
\[
\begin{pmatrix}
-1 - \sqrt{2} & -1 - \sqrt{2} & -1            & 0             & -1 - \sqrt{2} & 1 + \sqrt{2}\\
0             &  -1           & -1 - \sqrt{2} & -1 - \sqrt{2} & 0             & -3-2\sqrt{2} \\
1 + \sqrt{2}  & -1 - \sqrt{2} & 3+2\sqrt{2}   & 0             & 1 + \sqrt{2}  & 1 + \sqrt{2} \\
0             &  3+2\sqrt{2}  & -1 -\sqrt{2}  & 1 + \sqrt{2} & 0             & 1 \\
0             & 3+2\sqrt{2}   & -1 -\sqrt{2}  & 1 + \sqrt{2} & 0             & 1 \\
-3-2\sqrt{2}  & 0             & 0             & -3 -2\sqrt{2} & 1            & 0 \\
0             & -3-2\sqrt{2} & 1 + \sqrt{2} & 1 + \sqrt{2} & 0             & -1 \\
3 + 2\sqrt{2} & 0              & 0            &  1           & -1            & 0 \\
-1 - \sqrt{2} & 1 + \sqrt{2}   & 1            & 0            & -1 - \sqrt{2} & -1 - \sqrt{2} \\
0             & -3-2\sqrt{2} & 1 + \sqrt{2} & 1 +\sqrt{2}  & 0             & -1 \\
1+\sqrt{2}    & 1 + \sqrt{2}   & -3-2\sqrt{2} & 0             & 1 + \sqrt{2} & -1 - \sqrt{2} \\
0             & 1              & 1 + \sqrt{2} & -1 - \sqrt{2} & 0            & 3 + 2\sqrt{2}
\end{pmatrix}.
\]
As can be checked either by hand or in any computer algebra system, the above matrix has full rank. Its transpose is therefore injective, proving that the Bolza surface is perfect.

We saw earlier that the Bolza surface is also eutactic. Hence, it is a strict local maximum for $\ELsys$ by \cite[Proposition 2.1]{Bav:97}. Together with \corref{cor:roottwo}, this proves \thmref{thm:main}.

\subsection{The face curves}

By \corref{cor:spectrum}, the extremal length systole of $\calO$ is realized by the face curves, of which there are only four. It follows that $\calO$ is not perfect, since the dimension of $T_\calO \calT(S_{0,6})$ is equal to $6$. If extremal length was convexoidal, then we could conclude that $\calO$ is not a local maximizer for the extremal length systole by \cite[Proposition 2.3]{Bav:97}. 

Since $\calO$ is eutactic, there does not exist a tangent vector in the direction of which the extremal length of each face curve has a positive derivative. To determine if $\calO$ is a local maximizer for $\ELsys$ would therefore require estimating extremal length up to order two. Theorem 1.1 in \cite{LiuSu} shows that the sum of the second derivative of the extremal length of a curve along the Weil--Petersson geodesics in two directions $v$ and $iv$ is positive, but this could still allow one of them to be negative.

A potentially interesting deformation for disproving local maximality would be to twist two opposite faces with respect to each other and push them towards each other (at a slower rate). This should correspond to the direction $iv$ where $v$ is the gradient of the extremal length of the associated face curve. The extremal length systole of the square pillowcase (or the square torus) can be increased in that way. If such a deformation increased the extremal length systole of the octahedron, then one would expect to reach a local maximum once the opposite faces are aligned, that is, at a right triangular prism with equilateral base. However, the numerical calculations carried out in Appendix \ref{sec:prisms} indicate that all such prisms have extremal length systole at most $2.6236<\ELsys(\calO)$. Furthermore, the regular octahedron is an antiprism like the regular tetrahedron, which maximizes the extremal length systole in its moduli space by \corref{cor:tetrahedron}.

We therefore conjecture that $\calO$ maximizes the extremal length systole among all six-times-punctured spheres. This would imply that Voronoi's criterion fails for the extremal length systole and hence that extremal length is not convexoidal with respect to any connection on Teichm\"uller space.

%% file: landen.tex
\section{A geometric proof of the Landen transformations}  \label{sec:Landen}

There are many known proofs of the Landen transformations \cite{LandenSurvey}. Although the proof we give below can be reformulated as a change of variable, it at least explains where the latter comes from.

\begin{thm} \label{thm:landen}
For any $k \in (0,1)$, we have
\[
K(k) = \frac{1}{1+k}\, K(k^*) \quad \text{and} \quad K'(k) = \frac{2}{1+k} \,K'(k^*)
\]
where $k^* = 2 \sqrt{k} / (1+k)$.
\end{thm}
\begin{proof}
We start by proving the multiplication rule $2K(k)/K'(k) = K(k^*)/K'(k^*)$.

Recall that by choosing an appropriate branch of the square root in each half-plane, the Schwarz--Christoffel transformation
\[
T(\zeta) = \int_{\zeta_0}^\zeta \frac{dz}{\sqrt{(1-z^2)(1-(k^*)^2z^2)}}
\]
sends $\CHAT \setminus \{ -1/k^*, -1, 1, 1/k^* \}$ conformally onto a rectangular pillowcase $P(k^*)$ of width $2K(k^*)$ and height $K'(k^*)$, with the punctures mapping to the vertices. By symmetry, $T(0)$ is the midpoint of the bottom side of $P(k^*)$ and $T(\infty)$ is the midpoint of the top side.

Clearly, $T$ conjugates the action of $z \mapsto -z$  on $\CHAT \setminus \{ -1/k^*, -1, 1, 1/k^* \}$ with the rotation $J$ of angle $\pi$ around the vertical axis through $T(0)$ and $T(\infty)$ if we think of $P(k^*)$ as sitting upright in $\R^3$ (see \figref{fig:skewer}). That is, $J$ swaps the front and back faces of $P(k^*)$ and preserves the bottom and top edges.

\begin{figure}[htp]
    \centering
    \includegraphics[height=2in]{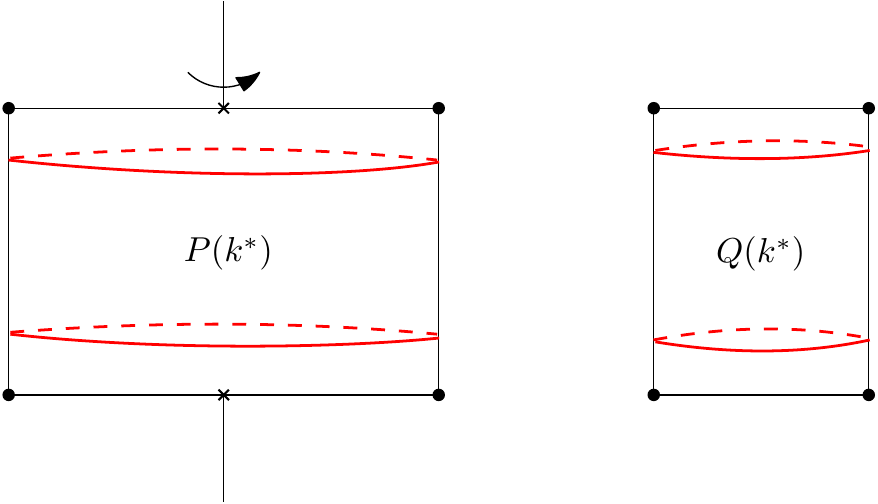}
    \caption{The hyperelliptic involution $J$ of the pillowcase $P(k^*)$ and the quotient $Q(k^*)$}
    \label{fig:skewer}
\end{figure}

To quotient by the action of $z\mapsto -z$, we apply the squaring map and puncture at the critical values, resulting in $\CHAT \setminus \{0,1, 1/(k^*)^2, \infty \}$. On the other hand, quotienting $P(k^*)$ by $J$ gives a pillowcase $Q(k^*)$ of half the width $K(k^*)$ and the same height $K'(k^*)$. The transformation $T$ descends to a conformal map between these two objects.

Observe that the M\"obius transformation
\[
g(z) = \left( \frac{1+k}{2} \right) \frac{1+z}{1+kz}
\]
sends $-1/k$, $-1$, $1$, and $1/k$ to $\infty$, $0$, $1$, and $1/(k^*)^2$ respectively. Another Schwarz-Christoffel transformation sends $\CHAT \setminus \{-1/k,-1,1,1/k \}$ to a pillowcase $P(k)$ of width $2K(k)$ and height $K'(k)$. Since the pillowcase representation of a four-times-punctured sphere is unique up to scaling, we have that $P(k)$ and $Q(k^*)$ have the same aspect ratio. That is,
\[
2K(k)/K'(k) = K(k^*)/K'(k^*).
\]
This implies that $K(k) = \lambda(k) K(k^*)$ and $K'(k) = 2 \lambda(k) K'(k^*)$ for some $\lambda(k)>0$. 

To determine this scaling factor, define
 \[
q=\frac{dz^2}{(1-z^2)(1-(k^*)^2z^2)}, \quad \omega = \frac{dz^2}{4z(1-z)(1-(k^*)^2 z)} \quad \text{and} \quad \tau = \frac{dz^2}{(1-z^2)(1-k^2z^2)}.
\]
If $f(z)=z^2$, then $f^*\omega = q$ and similarly $g^*\omega = \left( \frac{1+k}{2} \right)^2 \tau$ (these elementary calculations are left to the reader).

We then have 
\[
K(k^*) = \int_0^1 \sqrt{q} = \int_0^1 \sqrt{f^* \omega} = \int_0^1 \sqrt{\omega}\]
and
\[
K(k) = \int_0^1 \sqrt{\tau} = \frac{2}{1+k} \int_0^1 \sqrt{g^* \omega} = \frac{1}{1+k} \int_{-1}^1 \sqrt{g^* \omega} =
\frac{1}{1+k} \int_{0}^1 \sqrt{\omega} = \frac{1}{1+k} K(k^*).
\]
It follows that $K'(k) = (2K(k)/K(k^*)) K'(k^*) = 2K'(k^*)/(1+k)$.
\end{proof}

%% file: prisms.tex
\section{Prisms and antiprisms} \label{sec:prisms}

In this section, we obtain upper bounds on the extremal length systole of six-times-punctured spheres with $D_3$-symmetry. These estimates indicate that the regular octahedron has maximal extremal length systole among all prisms and antiprisms, which leads us to think that it is the global maximizer despite not being perfect.

\subsection{Antiprisms}
Within the $1$-parameter family of antiprisms \[\calA_r = \CHAT \setminus \{ 1, e^{2\pi i / 3}, e^{-2\pi i / 3}, -r , re^{\pi i / 3}, r e^{-\pi i / 3}  \}\]  where $r\geq 1$, it is easy to see that the regular octahedron $\calO \cong \calA_{2+\sqrt{3}}$  locally maximizes the extremal length systole. Indeed, the extremal length of the central face curve surrounding the cube roots of unity in $\calA_r$ has a strictly negative derivative in the $\del / \del r$ direction, because this pushes the three punctures furthest away from the origin exactly in the direction where the residue of the associated quadratic differential is positive (the ``horizontal'' direction at these poles). Alternatively, this can be shown by applying the cubing map as in the proof of \propref{prop:zigzag} to get that the extremal length is exactly $6K(1/\sqrt{1+r^3})/K'(1/\sqrt{1+r^3})$. This ratio has a strictly negative derivative. Since $\calO$ is eutactic, the derivative of the extremal length of the other face curves at $r=2+\sqrt{3}$ must be negative in the direction $-\del / \del r$ (by rotational symmetry, the three other face curves have the same extremal length). Thus, the directional derivative of the extremal length systole is negative in both directions $\del / \del r$ and $-\del / \del r$ at the regular octahedron.

The other remarkable surface in the family of antiprims is $\calA_1$, which is conformally equivalent to the double of a regular hexagon.

\begin{prop}
The extremal length systole of $\calA_1$ is at most \[
4K(w)/K'(w) \in  \left[2.34031875460627 \pm 5.71\cdot 10^{-15} \right]\]
where $w=2-\sqrt{3}$.
\end{prop}
\begin{proof}
We begin by applying the Cayley transform $z\mapsto i(z-i)/(z+i)$ to send the unit circle to the real line. The image of $\calA_1$ is $\CHAT \setminus \left\{\pm (2-\sqrt{3}), \pm  1 ,\pm (2+\sqrt{3})\right\}$. Then scale by $(2+\sqrt{3})$ to obtain $Z=\CHAT \setminus \left\{\pm 1, \pm  (2+\sqrt{3}) ,\pm (2+\sqrt{3})^2 \right\}$. To compute the extremal length of the curve $\alpha$ surrounding $[-1,1]$, we apply the squaring map $f(z)=z^2$ and puncture at its critical values to obtain $\CHAT \setminus \{ 0, 1 , (2+\sqrt{3})^2, (2+\sqrt{3})^4,\infty \}$. Then $f(\alpha)=\beta^2$ for the simple closed curve $\beta$ surrounding $[0,1]$. Furthermore, the extremal length of $\beta$ is the same in \[\CHAT \setminus \{ 0, 1 , (2+\sqrt{3})^2, (2+\sqrt{3})^4, \infty \}\] as in $W=\CHAT \setminus \{ 0, 1 , (2+\sqrt{3})^2, \infty \}$ because $(2+\sqrt{3})^4$ lies on one of the critical horizontal trajectories of the extremal differential on $W$.

By the proof of \thmref{thm:landen}, there is a M\"obius transformation sending $0$, $1$, $(2+\sqrt{3})^2$, and $\infty$ to $-1$, $1$, $1/k$, and $-1/k$ for the unique $k\in (0,1)$ such that $1/(k^*)^2 =(2+\sqrt{3})^2$. This gives $k^* = 1/(2+\sqrt{3}) = 2-\sqrt{3}$. Then
\[
\EL(\alpha,Z)= \EL(f^{-1}(\beta),Z)= 2\EL(\beta,W\setminus\{(2+\sqrt{3})^4\}) = 2\EL(\beta,W)= \frac{8K(k)}{K'(k)} = \frac{4K(k^*)}{K'(k^*)}.
\]
This last ratio belongs to the interval $\left[2.34031875460627 \pm 5.71\cdot 10^{-15} \right]$.
\end{proof}

In particular, $\ELsys(\calA_1)<\ELsys(\calO)$. With similar techniques as in \secref{sec:estimates}, it is possible to show that the shortest curves in $\ELsys(\calA_1)$ are the six ``edge curves'' whose extremal length was computed in the above proposition. As $r$ increases, we believe that their extremal length increases until at some point three face curves become shorter. Then $\ELsys(\calA_r)$ keeps increasing until $r$ reaches $2+\sqrt{3}$ where the fourth face curve has the same extremal length of the others.  After that point, the central face curve becomes shortest and its extremal length decreases to zero as $r \to \infty$. That is, we conjecture that $r \mapsto \ELsys(\calA_r)$ attains a unique local maximum at $r=2+\sqrt{3}$.

\subsection{Prisms}

The next interesting family of six-times-punctured spheres are the right tri\-an\-gular prisms with equilateral base punctured at their vertices. Every such prism is conformally equivalent to
\[
\calP_r := \CHAT \setminus \{ 1, e^{2\pi i / 3}, e^{-2\pi i / 3}, r , re^{2\pi i / 3}, r e^{-2\pi i / 3}  \}.
\]
for some $r>1$.

We start with a rigorous upper bound for the extremal length systole of $\calP_r$.

\begin{prop} \label{prop:upper}
The inequality $\ELsys(\calP_r) \leq 2\sqrt{3} \approx 3.464102$ holds for every $r>1$.
\end{prop}
\begin{proof}
Let $\alpha$ be the circle of radius $\sqrt{r}$ in $\calP_r$ and let $\beta=\beta_0 \cup \beta_1 \cup \beta_2$ where each $\beta_j$ surrounds the two punctures on the ray at angle $2\pi j$ from the positive real axis. The cubing map $f(z)=z^3$ sends $\calP_r$ to $\CHAT \setminus \{ 0 , 1 , r^3, \infty \}$ after puncturing at the critical values. Furthermore $f(\alpha)=\gamma^3$ and $f(\beta)=3\delta$ where $\gamma$ and $\delta$ surround $[0,1]$ and $[1,r^3]$ respectively.

Let $W$ and $H$ be the width and height of the pillowcase representation of $\CHAT \setminus \{ 0 , 1 , r^3, \infty \}$ where $[0,1]$ is horizontal. Then $\EL(\gamma) = 2 W / H$ and $\EL(\delta) = 2H/W$, so that \[\EL(\gamma)\EL(\delta)=4.\]

By \lemref{lem:branch}, we have $\EL(\alpha) = 3 \EL(\gamma)$. Furthermore, we claim that $\EL(\beta_0) \leq \EL(\delta)$. This is because the cylinder $C$ of circumference $2H$ and height $W$ for $\delta$ lifts under $f$ to a cylinder homotopic to $\beta_0$ (or any $\beta_j$). By monotonicity of extremal length under inclusion, we have $\EL(\beta_0)\leq \EL(C)= \EL(\delta)$.

We thus have
\[
\EL(\alpha)\EL(\beta_0) \leq 3\EL(\gamma)\EL(\delta) = 12,
\]
from which it follows that 
\[
\ELsys(\calP_r)\leq \min(\EL(\alpha),\EL(\beta_0)) \leq \sqrt{12} = 2 \sqrt{3}. \qedhere
\]
\end{proof}

The upper bound from \propref{prop:upper} can be improved to
\begin{equation}
\ELsys(\calP_r) \lesssim 2.6236 < 2.799 < \ELsys(\calO)
\end{equation}
for all $r>0$ using numerical calculations, as we now explain.

Let $x = \EL(\alpha,\calP_r)$. One can show that $x = 6K(r^{-3/2})/K'(r^{-3/2})$, so that the map $r\mapsto x$ is a strictly decreasing diffeomorphism from $(0,\infty)$ to itself. To obtain this formula, the idea is to apply the cubing map $f(z)=z^3$ and to puncture at $0$ and $\infty$. Then $f$ maps $\alpha$ to $\gamma^3$ where $\gamma$ is a simple closed curve surrounding the interval $[0,1]$ in $\CHAT \setminus \{ 0, 1, r^3 , \infty \}$.  By \lemref{lem:elliptic}, the extremal metric for $\gamma$ is a rectangular pillowcase. Therefore, the extremal metric $\rho$ for $\alpha$ is the triple branched cover of this pillowcase branched over two punctures that lie above each other. This means that $\rho$ looks like the 2-sided surface of the cartesian product between a tripod and an interval (see \figref{fig:tripod}). If we scale the metric so that the height is equal to $1$, then each leg of the tripod has length $x/6$, so that the total circumference is $x$. The cylinder $C$ evoked in the proof of \propref{prop:upper} simply goes around one of the pages of this open book in the vertical direction. To get a better estimate for $\EL(\beta_0)$, it suffices to find a larger embedded annulus in the same homotopy class. We construct one using the flat metric $\rho$.

\begin{figure}[htp]
     \centering
     \includegraphics[scale=1]{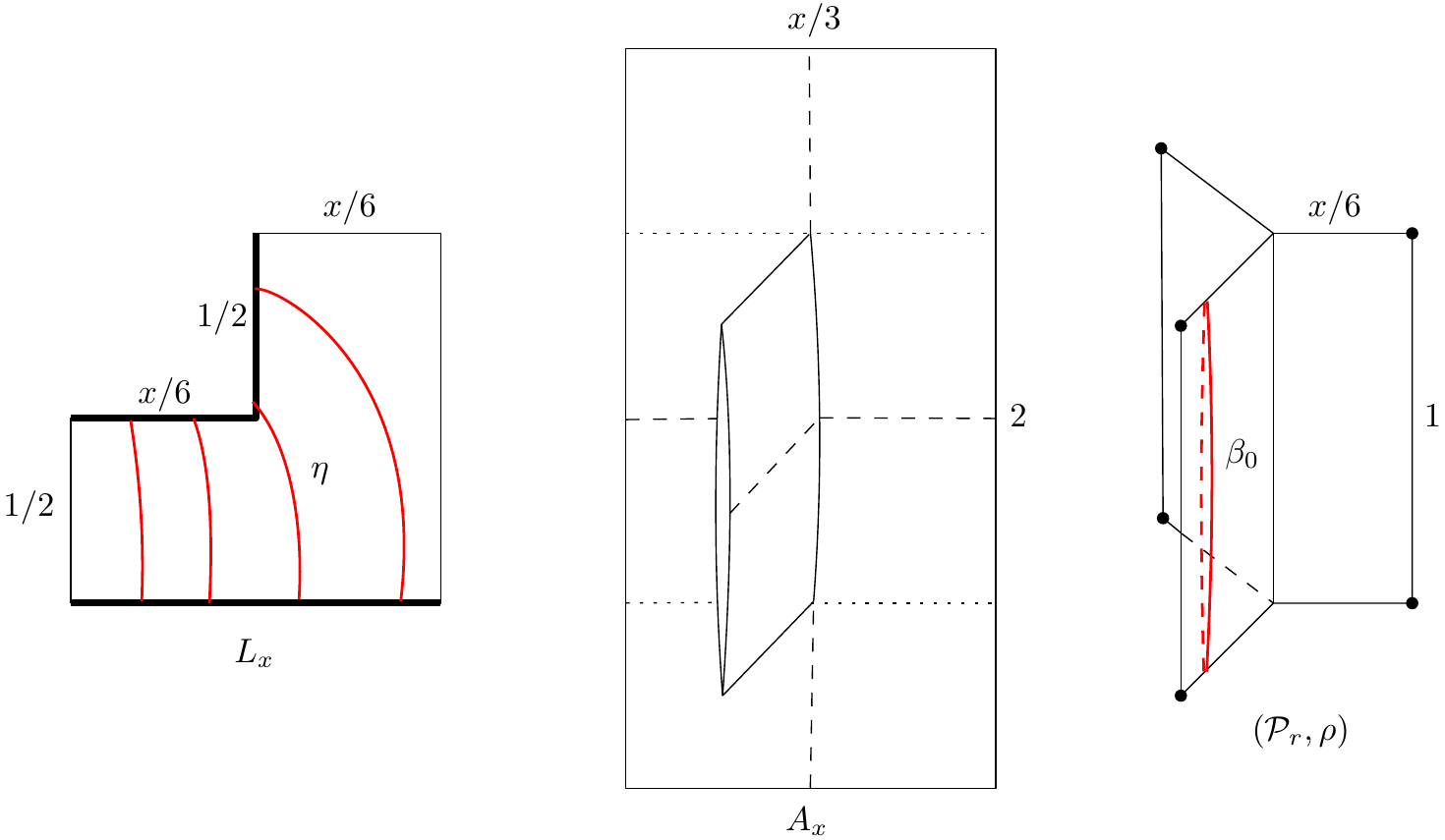}
    \caption{The polygon $L_x$, the annulus $A_x$, and the flat metric $\rho$ on $\calP_r$. The embedding $A_x^\circ \hookrightarrow \calP_r$ is obtained by folding the flaps at the top and bottom of the open annulus $A_x^\circ$ along the dotted lines.}
     \label{fig:tripod}
\end{figure}

Consider the polygon $L_x \subset \C$ with vertices at $0$, $x/3$, $x/3+i$, $x/6+i$, $x/6+i/2$, and $i/2$, as depicted in \figref{fig:tripod}. Reflect $L_x$ across the real axis, then double the resulting $T$-shape across the two pairs of sides that form an interior angle of $3\pi/2$. The resulting object $A_x$ is a topological annulus. Geometrically, it is equal to the cylinder $C$ from before glued onto an $x/3$ by $2$ rectangle via a vertical slit in the center. The interior $A_x^\circ$ embeds isometrically into $\calP_r$ equipped with the metric $\rho$ (see \figref{fig:tripod}), with its core curve mapping to $\beta_0$. By monotonicity of extremal length under inclusion, we have $\EL(\beta_0, \calP_r) \leq \EL(\beta_0,A_x^\circ)$. Furthermore, a symmetry argument implies that $\EL(\beta_0,A_x^\circ) = 4 \EL(\eta, L_x)$ where $\eta$ is the set of all arcs joining the bottom side of $L_x$ to the pair of sides forming an interior angle of $3\pi/2$. We thus obtain
\[
\ELsys(\calP_r) \leq \min(\EL(\alpha,\calP_r), \EL(\beta_0,\calP_r)) \leq \min(x, 4\EL(\eta, L_x)).
\]

\begin{figure}[htp]
    \centering
    \includegraphics[scale=1]{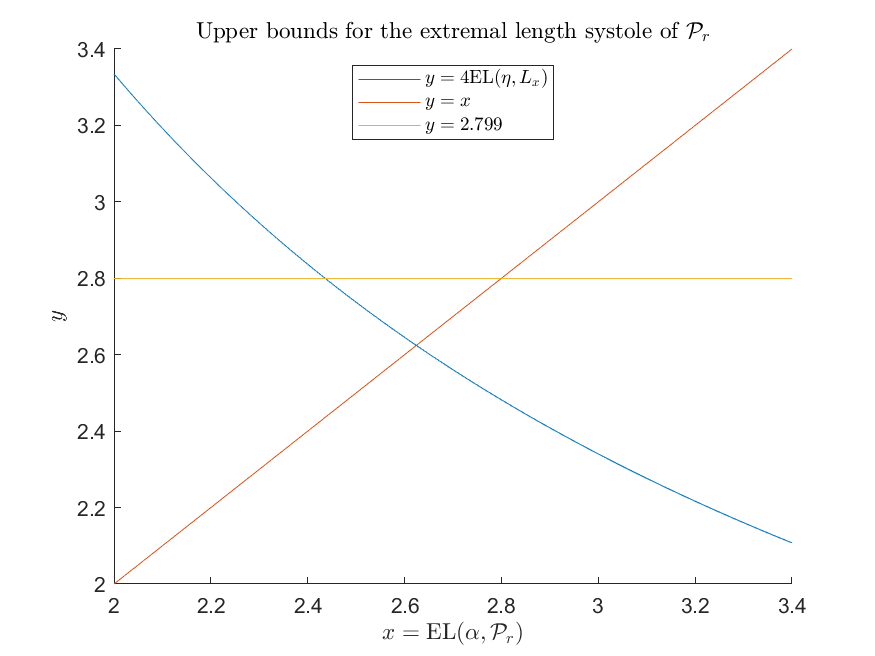}
    \caption{Numerical estimates for $\ELsys(\calP_r)$}
    \label{fig:plot}
\end{figure}

This last extremal length $\EL(\eta, L_x)$ can be calculated by finding a conformal map from $L_x$ onto a rectangle $R_x$, with the ``sides'' of $L_x$ joined by $\eta$ mapping to the vertical sides of $R_x$. We carried out this computation for one thousand equally spaced values of $x$ in the interval $[2,3.4]$ using the Schwarz--Christoffel Toolbox \cite{Driscoll} for MATLAB \cite{Matlab}. The resulting bounds for $\ELsys(\calP_r)$ are shown in \figref{fig:plot}. Note that $\EL(\eta, L_x)$ is strictly decreasing in $x$, so the maximum of the upper bound $\min(x,4\EL(\eta, L_x))$ is achieved where $x = 4\EL(\eta, L_x)$, which occurs around $x\approx2.6236$ according to our numerical calculations.

Although the Schwarz--Christoffel Toolbox does not come with certified error bounds, it is quite reliable especially for the range of polygons we consider, where crowding of vertices does not occur. One could turn this into a rigorous upper bound using similar methods as in \cite[Section 6]{nonconvex}.


